\documentclass[a4paper,11pt]{article}

\usepackage{amsmath}
\usepackage{amssymb}
\usepackage{amsthm}
\usepackage{amsfonts}
\usepackage{mathtools}
\usepackage{graphicx}
\usepackage{subfig}

\usepackage{titlesec}
\usepackage{enumitem}
\usepackage{mathtools}
\usepackage{listings}
\usepackage{geometry}
\usepackage{bbm}
\usepackage{xcolor}
\usepackage{hyperref}
\usepackage{caption}
\usepackage{mathrsfs}
\usepackage{booktabs}
\usepackage{floatrow}
\newfloatcommand{capbtabbox}{table}[][\FBwidth]
\usepackage{blindtext}
\usepackage{colortbl}
\newtheorem{theorem}{Theorem}[section]
\newtheorem{corollary}[theorem]{Corollary}
\newtheorem{proposition}[theorem]{Proposition}
\newtheorem{lemma}[theorem]{Lemma}
\newtheorem{remark}[theorem]{Remark}
\newtheorem{definition}[theorem]{Definition}

\newcommand{\re}{\mbox{\rm Re}}
\newcommand{\im}{\mbox{\rm Im}}

\newcommand{\0}{L^2(\D)}
\newcommand{\1}{H^1(\D)}
\newcommand{\2}{H^{-1}(\D)}
\newcommand{\3}{H^{1}_{0}(\D)}
\newcommand{\4}{\tangentspace{\ci u}}
\newcommand{\5}{\mathscr{H}_{u}^{\mathbb{C} , \perp }}

\newcommand{\D}{\mathcal{D}}

\newcommand{\lm}{\lambda}

\newcommand{\R}{\mathbb{R}}
\newcommand{\C}{\mathbb{C}}

\newcommand{\Ltwo}[2]{(#1,#2)_{L^2(\D)}}
\newcommand{\Hone}[2]{(#1,#2)_{H^1(\D)}}

\definecolor{mycolor1}{rgb}{0.00000,0.44700,0.74100}%

\newcommand{\quotes}[1]{``#1''}

\newcommand{\dx}{\hspace{2pt}\mbox{d}x}

\newcommand{\ci}{\mathrm{i}} 

\newcommand{\sR}{\mbox{\rm \tiny R}}

\newcommand{\nablaR}{\nabla_{\hspace{-2pt}\sR}}
\newcommand{\VR}{V_{\hspace{-2pt}\sR}}

\newcommand{\tangentspace}[1]{T_{#1}\mathbb{S}}

\newcommand{\Acal}{\mathcal{A}}
\newcommand{\Ical}{\mathcal{I}}
\newcommand{\Jcal}{\mathcal{J}}

\begin{document}

\begin{center}{On discrete ground states of
rotating \\
Bose--Einstein condensates\renewcommand{\thefootnote}{\fnsymbol{footnote}}\setcounter{footnote}{0}
 \hspace{-3pt}\footnote{The authors acknowledge the support by the German Research Foundation (DFG grant HE 2464/7-1).}}\\[2em]
 \end{center}

\begin{center}
{\large Patrick Henning\footnote[1]{\label{affiliation}Department of Mathematics, Ruhr-University Bochum, DE-44801 Bochum, Germany.\\ email: \href{mailto:patrick.henning@rub.de}{patrick.henning@rub.de} and \href{mailto:mahima.yadav@rub.de}{mahima.yadav@rub.de} }
and Mahima Yadav\textsuperscript{\ref{affiliation}}}\\[2em]
\end{center}

\begin{center}
{\large{\today}}
\end{center}

\begin{center}
\end{center}

\begin{abstract}
The ground states of Bose--Einstein condensates in a rotating frame can be described as constrained minimizers of the Gross--Pitaevskii energy functional with an angular momentum term. In this paper we consider the corresponding discrete minimization problem in Lagrange finite element spaces of arbitrary polynomial order and we investigate the approximation properties of discrete ground states. In particular, we prove a priori error estimates of optimal order in the $L^2$- and $H^1$-norm, as well as for the ground state energy and the corresponding chemical potential. A central issue in the analysis of the problem is the missing uniqueness of ground states, which is mainly caused by the invariance of the energy functional under complex phase shifts. Our error analysis is therefore based on an Euler--Lagrange functional that we restrict to certain tangent spaces in which we have local uniqueness of ground states. This gives rise to an error decomposition that is ultimately used to derive the desired a priori error estimates. We also present numerical experiments to illustrate various aspects of the problem structure.
\end{abstract}

\section{Introduction}

At extreme temperatures close to 0 Kelvin, dilute bosonic gases can condensate to a fascinating state of matter: a so-called Bose-Einstein condensate (BEC), cf. \cite{Bos24,Ein24,PiS03}. In such a condensate most of the particles occupy the same quantum state which makes it effectively behave like a macroscopically observable \quotes{super atom}. This super atom allows to study various quantum phenomena, where superfluidity, i.e., frictionless flow, is perhaps one of most intriguing ones \cite{MAH99}. Superfluidity can be verified by rotating the BEC with a stirring potential, triggering  the formation of vortices with a quantized circulation if the rotation frequency is appropriately chosen.

The ground states of a rotating BEC can be described as the global minimizers $u \in H^1_0(\D,\C)$ of the Gross--Pitaevskii energy functional 
\begin{align*}
E(w)= \frac{1}{2}\int_{\mathcal{D}} |\nabla w|^2 + V\, |w|^2  - \Omega\, \bar{w}\, \mathcal{L}_{3}w + \frac{\beta}{2} |w|^4 \dx
\end{align*}
under the constraint that $\int_{\D}|w|^2 \dx =1$. Here, $\D \subset \R^d$ (for $d=2,3$) is the computational domain, $V \in L^{\infty}(\D,\R_{\ge 0})$ models a trapping potential, $\beta \in \R_{\ge }0$ describes repulsive particle interactions, $\Omega \in \R$ is the angular velocity and $\mathcal{L}_3 = - \ci \left( x_1 \partial_{x_2} - x_2 \partial_{x_1} \right)$ the $x_3$-component of the angular momentum. The existence and non-existence of ground states in the above setting are e.g. studied in \cite{BWM05}.

In practice, a numerical scheme for computing ground states involves two steps: 1. a space discretization which essentially requires a minimization of $E$ over a finite dimensional (discrete) space and 2. an appropriate iterative solver that allows to find the corresponding discrete minimizers. We will not discuss the second step here but just refer to \cite{AHP21NumMath,AltPetSty22,AnD14,ALT17,BaC13b,BCW10,BaD04,BaS08,BWM05,COR09,CaL00,DaP17,DiC07,HeP20,JarKM14} and the references therein for different approaches and complementary analytical results (as far as they exist). The focus of this paper is on the first step, i.e., the approximation of analytical ground states in discrete spaces and a corresponding error quantification. Typical space discretizations involve spectral and pseudo-spectral methods \cite{BaC13b,CCM10}, quasi-uniform finite element spaces \cite{CCM10} or adaptive finite elements \cite{CGHZ11,DaH10,HSW21}. Two-grid post-processing techniques in FE spaces are presented in \cite{CCH18,HMP14b}.

When it comes to the analysis of these space discretizations, and in particular the derivation of a priori error estimates, the case of rotating Bose--Einstein condensates has not been regarded yet. For non-rotating condensates (i.e. for $\Omega=0$) the first results were obtained by Zhou \cite{Zho04,Zho07} who considered arbitrary finite dimensional spaces and derived corresponding $H^1$-error estimates. The estimates were however not sharp yet due to a suboptimal contribution from the $L^2$-norm of the error. This gap could be closed by Canc\`es et al. \cite{CCM10} who revisited the problem and provided the first sharp error analysis, covering $L^2$- and $H^1$-error estimates together with estimates for the ground state energy and the so-called ground state eigenvalue. Shortly after, the results were generalized yet another time by Zhou and coworkers \cite{CHZ11} by adding the effect of Hartree-type potentials to the analysis. Error estimates for finite element discretizations of the Kohn-Sham and Hartree-Fock equation were obtained in \cite{CDGHZ14,CGHY13,MadayTurinici2000} and generalized finite element approximations for the standard Gross--Pitaevskii model were studied in \cite{HMP14b,HeP21}.

Extending the aforementioned error analysis to rotating BECs is however not straightforward. This is because the setting for $\Omega=0$ is fully real-valued, whereas the problem becomes complex-valued for $\Omega \not=0$. As a result, uniqueness of ground states is lost, cf. \cite{BWM05}. One particular issue are complex phase shifts of the form $e^{\ci \theta}$ for angles $\theta \in [0,2\pi)$, i.e., if $u$ is an $L^2$-normalized minimizer of $E$ then $e^{\ci \theta} u$ is also a normalized minimizer for any $\theta$. This means that ground states are not locally unique. Our analysis is therefore based on restricting the Euler-Lagrange functional for the constrained minimization problem to tangent spaces where the second Fr\'echet derivative of the Euler-Lagrange functional (evaluated at a ground state) has a bounded inverse. On these spaces, we can derive $H^1$-error estimates by applying abstract approximation results for discrete solutions to nonlinear equations. For estimating the $L^2$- and eigenvalue errors we change the approach and draw inspiration from the Kohn-Sham setting \cite{CDGHZ14,CGHY13,MadayTurinici2000} (where also missing local uniqueness is faced) and consider an $L^2$-orthogonal decomposition of the error with contributions from the tangent spaces. This decomposition gives rise to strong error identities that we then use to derive different types of error estimates in $\mathbb{P}^k$-Lagrange finite element spaces $V_{h,k}$. In particular, we obtain convergence of order $\mathcal{O}(h^{k})$ for the $H^1$-error and convergence of order $\mathcal{O}(h^{k+1})$ for the $L^2$-error. We also derive optimal order estimates for the ground state energy and the ground state eigenvalue, both of order $\mathcal{O}(h^{2k})$. With this, our findings directly generalize the results of \cite{CCM10} to the Gross--Pitaevskii equation with rotation and to polynomial spaces of higher order.\\[0.5em]
{\it Outline:} The paper is structured as follows. In Section \ref{section-setting}, we give an introduction to the analytical setting of the Gross-Pitaevskii equation with angular momentum, we discuss the issue of uniqueness and state our basic assumptions. The finite element discretization and our main results regarding the corresponding approximation properties are presented in Section \ref{fem-main-results} and afterwards illustrated by numerical experiments in Section \ref{exp-second-derivative}. Section \ref{proof-of-approx-theorem} is devoted to the proofs of the main results: Sections \ref{basic_analysis} and \ref{subsection:abstracy-bounds} provide an abstract asymptotic convergence analysis together with some preliminary estimates that allow to control energy- and eigenvalue-errors in terms of the $H^1$-error. The $H^1$-error itself is then estimated in Section \ref{subsection-H1-estimate}, whereas the $L^2$-error estimate follows in Section \ref{subsection-L2-estimate}. An error estimate of optimal order for the eigenvalue is then proved in Section \ref{proof-of-order-theorem}. Complementary results on the inf-sup stability are given in the appendix.

\section{Analytical setting}
\label{section-setting}

In the following we shall consider the Gross-Piteavskii equation on a computational domain $\D$ with homogeneous Dirichlet boundary conditions, where we assume that
\begin{enumerate}[label={(A\arabic*)}]
\item \label{A1} $\D \subset \mathbb{R}^d$ is a bounded, convex domain for $d=2,3$ with polygonal boundary.
\end{enumerate}
Since we are concerned with the minimization of a real-valued energy functional $E$, it is natural to consider all Hilbert spaces as {\it real-linear} spaces. To be precise, we interpret the Lebesgue space $L^2(\mathcal{D}):= L^2(\mathcal{D},\C)$ as a {\it real} Hilbert space equipped with the inner product 
$$ 
\Ltwo{v}{w}:= \re \Big( \int_{\D} v \, \overline{w} \dx \Big),
$$
where $\overline{z}$ denotes the complex conjugate of a number $z\in \C$. Analogously, the Sobolev space $H_{0}^{1}(\mathcal{D}):= H_{0}^{1}(\mathcal{D},\mathbb{C})$ is also seen as a real Hilbert space equipped with the inner product
$$\hspace{6mm}
\Hone{v}{w}:= \re \Big(  \int_{\mathcal{D}} v \, \overline{w} \dx + \int_{\mathcal{D}} \nabla v \cdot \overline{\nabla w} \dx \Big).$$
The corresponding $(\text{real})$ dual space is denoted by $ \big(H_{0}^{1}(\D)\big)^{*} = H^{-1}(\mathcal{D}):= H^{-1}(\mathcal{D},\C)$, with $ \langle \cdot , \cdot \rangle := \langle \cdot ,\cdot \rangle _{H^{-1}(\mathcal{D}) , H^{1}_{0}(\mathcal{D}}$ being the canonical duality pairing in $H_{0}^{1}(\mathcal{D})$. 

For the data, we assume that
\begin{enumerate}[resume,label={(A\arabic*)}]
\item\label{A2} $V \in L^{\infty}(\mathcal{D},\mathbb{R}_{\ge 0} )$ denotes a non-negative potential and
\item\label{A3} the parameter $\beta \in \R_{\ge0}$ characterizes the strength of (repulsive) particle interactions.
\end{enumerate}
To model condensates that are rotating around the $x_3$-axis of the coordinate system, we require the third component of the angular momentum operator, which we denote by $\mathcal{L}_3 := - \ci \left( x_1 \partial_{x_2} - x_2 \partial_{x_1} \right)$. The angular velocity of that component shall be given by $\Omega \in \mathbb{R}$. To guarantee existence of ground states, we formally need to add a balancing assumption that the trapping frequencies in $x_1$- and $x_2$-direction are larger than the angular frequency, i.e.
\begin{enumerate}[resume,label={(A\arabic*)}]
\item\label{A4} there is $\varepsilon > 0$ such that
\begin{align*}
	V(x) - \frac{1 + \varepsilon}{4} \Omega^2 (x_1^2 + x_2^2) \ge  0 \quad \text{for almost all } x \in \D.
\end{align*}
\end{enumerate}
Physically speaking, assumption \ref{A4} ensures that centrifugal forces do not become too strong in comparison to the strength of the trapping potential $V$. 

With the above notation, we can define the Gross--Pitaevskii energy functional
$$
E :H_{0}^{1}(\D) \rightarrow \mathbb{R}
$$
for $w\in H^1_0(\D)$ by
\begin{align}
\label{energy1}
E(w)
:= \frac{1}{2}\int_{\mathcal{D}} |\nabla w|^2 + V\, |w|^2  - \Omega\, \bar{w}\, \mathcal{L}_{3}w + \frac{\beta}{2} |w|^4 \dx.
\end{align}
Since ground states will be defined as constrained minimizers of $E$, we briefly need to discuss the well-posedness.
We start with the observation that $E$ is indeed real-valued and bounded from below: Using assumption \ref{A4}, we let $\VR(x):= V(x) - \tfrac{1}{4}\, \Omega^2\, |x|^2 \ge 0$ denote a {\it non-negative} modified potential and we let the covariant gradient $\nablaR$ be given by
\begin{align}
\label{def-nablaR}
\nablaR w := \nabla w + \ci \tfrac{\Omega}{2} R^{\top} w
\end{align}
for the divergence-free vector field $R(x):=(x_2,-x_1,0)$ if~$d=3$ and $R(x):=(x_2,-x_1)$~if $d=2$. With this
\begin{align}
\label{R-inner-product}
(v,w)_{\sR} := \re \Big(  \int_{\mathcal{D}} \nablaR v \cdot \overline{\nablaR w} \dx + \int_{\mathcal{D}} \VR v \, \overline{w} \dx  \Big) \, 
\end{align}
defines an inner-product on $H^{1}_{0}(\D)$ with induced norm $\| v \|_{\sR}:= \sqrt{(v ,v)_{\sR}}$. 
It is now easy to see (cf. \cite{DaK10,HeM17}) that $E(w)$ can be equivalently rewritten as
\begin{align}
\label{energy_function}
E(w) = \frac{1}{2} \| w \|_{\sR}^2 +  \frac{\beta}{4} \int_{\D} |w|^4 \dx. 
\end{align}
Hence, $E$ is a real functional that is bounded from below by zero.

A {\it ground state} of a Bose--Einstein condensate in the rotational frame is defined as a stable stationary state in the lowest possible energy level. Mathematically, it is characterized as a global minimizer $u$ of $E$ under the normalization condition that $\int_{\D} |u|^2 \dx =1$, i.e.
\begin{align}
\label{definition-groundstate}
E(u) =\underset{v \in \mathbb{S}}{\inf}\hspace{2pt} E(v)
\qquad \mbox{where } \mathbb{S}:= \{  v \in H^1_0(\D) \mbox{ and } \| v \|_{L^2(\D)} =1 \}.
\end{align}
The problem is a non-convex minimization problem on the $L^2$-sphere in $H^1_0(\D)$. Exploiting the reformulation \eqref{energy_function}, the energy functional is positively bounded from below on the manifold $\mathbb{S}$ with 
$$
E(v) \ge \tfrac{1}{2} \| v \|_{\sR}^{2} >  0 \qquad \mbox{for all } v\in \mathbb{S}.
$$
Additionally, $E$ is weakly lower semi-continuous on $H^1_0(\D)$ and it can be proved, cf. \cite[Lemma 2.2]{pc22}, that $\| \cdot \|_{\sR}$ is equivalent to the standard $H^{1}$-norm on $\3$. In particular, there exist constants $0<c(\varepsilon) < C(V,\Omega) < \infty$, depending on $V$, $\Omega$ and $\varepsilon$ respectively, such that
\begin{align}
\label{norm-equivalence-nablaR}
c(\varepsilon) \, \| v\|_{H^1(\D)}  \, \le \, \|  v \|_{\sR} \, \le \, C(V,\Omega)\,  \| v\|_{H^1(\D)}  \qquad \mbox{for all } v\in H^1_0(\D).
\end{align}
It can be shown that $c(\varepsilon)=\mathcal{O}(\sqrt{\varepsilon})$, i.e. the lower constant degenerates for $\varepsilon \rightarrow 0$, which corresponds to a violation of the physical assumption \ref{A4}. Now, since $\|  \cdot \|_{\sR}$ is equivalent to the standard $H^1$-norm on $H^1_0(\D)$, we easily see that $E$ fulfills the Palais--Smale condition on $\mathbb{S}$. As an even functional that fulfills the Palais--Smale condition we can apply classical Lusternik-Schnirelmann theory to $E$, cf. \cite[Chapt. 4, Theorem 5.5]{Kav93}, and conclude that $E$ has infinitely many critical values. The critical values are {\it unbounded from above} and bounded from below, where the smallest critical value is given by the ground state energy $E_0 \in \R$. In particular, there exists a sequence of critical values $E_k \in \R$ with $-\infty < E_0 \le E_1 \le E_2 \le \cdots E_k \rightarrow \infty$ for $k\rightarrow \infty$ and corresponding critical points $u_k\in\mathbb{S}$ with $E_k=E(u_k)$. A critical point to a smallest critical value $E_k=E_0$ is a ground state as defined in \eqref{definition-groundstate}. Note that a critical point $u_k \in \mathbb{S}$ to $E_k$ is not unique due to the invariance of $E$ under rotations of $u_k$ on the complex unit circle.
We summarize the previous discussion in the following corollary.
\begin{corollary}[Existence of ground states]
Under assumptions \ref{A1}-\ref{A4}, there exists at least one (physically meaningful) ground state $u\in \mathbb{S}$ to problem \eqref{definition-groundstate} and it holds $E(u)>0$.
\end{corollary}
\subsection{Uniqueness of ground states and quasi-isolation} 
\label{uniqueness-section}

Uniqueness of ground states can only be expected up to a constant phase factor $e^{\ci\theta}$ with angle $\theta \in [0,2\pi)$. This is because for any $u \in \mathbb{S}$ we have $e^{\ci\theta}u \in \mathbb{S}$ and $E(e^{\ci\theta}u)=E(u)$. Uniqueness of the ground state density $|u|^2$ (which is independent of $e^{\ci\theta}$) is possible, but can only be expected for frequencies $\Omega$ below a certain first critical frequency. For larger frequencies the ground state density $|u|^2$ can become non-unique, cf. \cite[Section 6.2]{BWM05}. The phenomenon is not yet fully understood, but numerical experiments and theoretical considerations indicate that missing uniqueness of $|u|^2$ is a rather rare occasion that only happens for certain  critical values, where a particular vortex state becomes energetically equally preferable to a ground state with less vortices (or without any vortices at all), cf. \cite{BWM05}. 
To account for the general lack of uniqueness, we define the set of ground states of \eqref{definition-groundstate} by
\begin{align*}
\mathcal{U}= \{ u \in \mathbb{S} \, |\, E(u) = \inf_{v \in \mathbb{S} } E(v)\}\, 
\end{align*}
and introduce the term of {\it quasi-isolated ground states}. Loosely speaking, we want that for every ground state $u \in \mathcal{U}$, there is a small $\delta$-neighborhood  $\mathbb{B}_{\delta}(u) := \{ v \in \mathbb{S} \,| \,\|v-u\|_{\1} < \delta \}$ of $u$ where all other ground states in  $\mathbb{B}_{\delta}(u)$ are precisely phase shifts of $u$, i.e.,
\begin{align}
\label{preliminary-quasi-isolation}
\hat{u} \in \mathbb{B}_{\delta}(u) \cap \mathcal{U} \quad \Rightarrow \quad \hat{u} = u e^{\ci \theta} \mbox{ for some } \theta \in [0,2 \pi).
\end{align}
This condition can be expressed through curves $\gamma: t \mapsto \gamma(t) \in \mathbb{S}$ crossing a ground state with $\gamma(0)=u \in \mathcal{U}$. Since $\gamma$ remains on $\mathbb{S}$, we have $\gamma^{\prime}(0) \in \tangentspace{u}$ with the tangent space in $u$ given by 
$$
\tangentspace{u} := \{ v \in H^1_0(\D) | \, (u,v)_{\0} = 0 \}.
$$
Now, if condition \eqref{preliminary-quasi-isolation} holds, the energy level $t \mapsto E(\gamma(t))$ has a strict local minimum in $t=0$ as long as $\gamma$ does not follow the direction $\ci u$ of complex phase shifts (since $\tfrac{\mbox{\scriptsize d}}{\mbox{\scriptsize d}t} e^{\ci t} u \vert_{t=0}=\ci u$).
This direction is blocked by imposing the condition
$\gamma^{\prime}(0) \in \tangentspace{\ci u}$.
Hence, for all curves with $\gamma^{\prime}(0) \in \tangentspace{u} \cap \tangentspace{\ci u}$, the necessary conditions for local minima imply $\tfrac{ \mbox{\scriptsize d}^{\,2} }{ \mbox{\scriptsize d}t^{2}} E(\gamma(t))\vert_{t=0} \ge 0$. The following definition formalizes these considerations, with the slight modification that the necessary condition for strict local minima is replaced by a sufficient condition. 
\begin{definition}[Quasi-isolated ground state] 
\label{definition-quasi-isolated-ground-state}
A ground state $u \in \mathcal{U}$ is called a \emph{quasi-isolated ground state} if for any differentiable curve $\gamma : (-\varepsilon,\varepsilon) \rightarrow \mathbb{S}$  with $\gamma(0)=u$ and $\gamma^{\prime}(0) \in (\tangentspace{u} \cap \tangentspace{\ci u}) \setminus \{ 0 \}$ it holds
 $$ \frac{ \mbox{d}^{\,2} }{ \mbox{d}\hspace{1pt}t^{2}} E(\gamma(t)) |_{t=0} > 0.$$
\end{definition} 
Since we will require the space $\tangentspace{u} \cap \tangentspace{\ci u}$ often, we introduce for brevity
\begin{align*}
\mathscr{H}_{u }^{\mathbb{C} , \perp}:=
\tangentspace{u} \cap \tangentspace{\ci u}.
\end{align*}
The notation is motivated by the simple observation that $\mathscr{H}_{u }^{\mathbb{C} , \perp}$ is the orthogonal complement of $u$ with respect to the {\it complex} $L^2$-inner product, whereas $\tangentspace{\ci u}$ is the orthogonal complement of $\ci u$ with respect to the {\it real} $L^2$-inner product.
With Definition \ref{definition-quasi-isolated-ground-state} we make the following assumption on the ground states:
\begin{enumerate}[resume,label={(A\arabic*)}]
\item\label{A5} Every ground state $u \in \mathcal{U}$ is a quasi-isolated ground state.
\end{enumerate}
With other words, two ground states are either well separated or complex phase shifts of each other. 

In general, assumption \ref{A5} is reasonable and its validity in a given setting can be always checked numerically. Even though the energy $E$ can be also invariant under rotations of the coordinate system in particular cases (such as for spherical domains $\D$ and radial symmetric potentials $V$), we can still expect \ref{A5} to hold. The reason is that the corresponding critical direction is typically not an element of $H^1_0(\D)$. For example, if $E(u)=E(u_{\theta})$ with $u_{\theta}(x):=u(\mathbf{R}_{\theta} x)$ for any rotation matrix 
$$
\mathbf{R}_{\theta} = \left(\begin{matrix}
\cos \theta &  -\sin \theta \\
\sin \theta & \cos \theta 
\end{matrix}\right),
$$
then the critical direction of the curve (on which the energy level would remain the same) is given by $\tfrac{\mbox{\scriptsize d}}{\mbox{\scriptsize d}\theta} u(\mathbf{R}_{\theta} x)_{\vert \theta = 0} =-(x \times \nabla)u$.  
Since we cannot expect $\nabla u$ to vanish on $\partial \D$, the function $-(x \times \nabla)u$ is not an element of $H^1_0(\D)$ and, hence, $\gamma^{\prime}(0)=-(x \times \nabla)u$ is not admissible. In fact, this reasoning also implies that $-(x \times \nabla)u$ cannot be an eigenfunction of $E^{\prime\prime}(u)-\lambda \mathcal{I}$ to the eigenvalue $0$ (see Proposition \ref{prop-inf-sup-stability-Wperpu} below for a corresponding connection).
\subsection{Characterization through the Gross--Pitaevskii equation}
Since every ground state $u$ is a critical point of $E$ on $\mathbb{S}$, we can characterize it by the corresponding Euler--Lagrange equations seeking $u\in \mathbb{S}$ with Lagrange multiplier $\lambda \in \R_{>0}$ such that
\begin{align*}
\langle E^{\prime}(u) , v \rangle = \lambda \, \langle M^{\prime}(u) , v \rangle  \qquad \mbox{for all } v\in H^1_0(\D),
\end{align*}
where $E^{\prime}: H^1_0(\D) \rightarrow H^{-1}(\D)$ denotes the Fr\'echet derivative of $E$ and $M: H^1_0(\D) \rightarrow \R$ is the scaled constraint functional given by $M(v):=\tfrac{1}{2}\left(\int_{\D}|v|^2 \dx - 1\right)$. Computing the Fr\'echet derivatives, we obtain
\begin{align*}
\langle M^{\prime}(u) , v \rangle  = \re \int_{\D} u \, \overline{v}\dx = (u,v)_{L^2(\D)}
\end{align*}
and
\begin{align*}
\langle E^{\prime}(u) , v \rangle 
&= (\nablaR  u , \nablaR v)_{\0} + \big((\VR  + \beta |u|^2)u,v\big)_{\0}.
\end{align*}
As a compact notation, we introduce for arbitrary $u\in H^1_0(\D)$, the $u$-linearized approximation of $E^{\prime}$ by $\Acal_{|u|}:H^1_0(\D) \rightarrow H^{-1}(\D)$ with 
\begin{align}
\label{defAcal-u-square}
\langle \Acal_{|u|}v , w \rangle := (\nablaR v , \nablaR w )_{\0} + \big((\VR  + \beta |u|^2)v,w\big)_{\0}
\qquad \mbox{for } v,w \in H^1_0(\D).
\end{align}
 Note that $\Acal_{|u|}$ is a self-adjoint operator on $L^2(\mathcal{D})$ and that it is obviously elliptic with
 \begin{align}
 \label{def-alpha}
 \langle \Acal_{|u|}v,v \rangle  \ge \alpha \, \| v\|_{H^1(\D)}^2 \quad \mbox{for all } v\in H^1_0(\D),
 \end{align}
 where $\alpha = c(\varepsilon)^2=\mathcal{O}(\varepsilon)$ is the square of the constant from \eqref{norm-equivalence-nablaR}.
As $E^{\prime}(v)=\Acal_{|v|}(v)$, we can write the Euler--Lagrange equations for the critical points of $E$ equivalently as: find $u\in \mathbb{S}$ and $\lambda >0$ such that
\begin{align}
 \label{eigen_value_problem_2}
\langle \Acal_{|u|}u , v \rangle = \lambda \, ( u , v )_{L^2(\D)} \qquad \mbox{for all } v\in H^1_0(\D).
\end{align}
The problem can be interpreted as an eigenvalue problem with eigenvector nonlinearity and is known as the {\it Gross--Pitaevskii eigenvalue problem} (GPEVP). In the sense of distributions, problem \eqref{eigen_value_problem_2} equivalently reads: find $\lambda > 0$ and $u \in H_{0}^{1}(\D)$ such that $ \int_{\D}|u|^2 dx =1$ and
\begin{align}
\label{eigen_value_problem_1}
-\Delta u + V\, u - \Omega \, \mathcal{L}_{3}u + \beta \, |u|^2u = \lambda \, u.
\end{align}
The equivalence of \eqref{eigen_value_problem_2} and \eqref{eigen_value_problem_1} is seen by noticing that the real part in front of the integrals can be dropped by using both $v$ and $\ci v$ as test functions.

Any eigenfunction to \eqref{eigen_value_problem_1} that is not a ground state is called a meta-stable {\it excited state} of the Bose--Einstein condensate.

\begin{remark}[Ground state eigenvalue]
An eigenvalue $\lambda$ that belongs to a ground state $u \in \mathbb{S}$ in the sense of \eqref{definition-groundstate} is called a \emph{ground state eigenvalue}. Note however, that it is in general not possible to directly use the GPEVP \eqref{eigen_value_problem_1} as a defining equation for ground states. The reason is that a ground state eigenvalue is not necessarily the smallest eigenvalue of the GPEVP \eqref{eigen_value_problem_1} and an eigenfunction to the smallest eigenvalue can potentially be an excited state, cf. \cite[Section 6.3]{AHP21NumMath}. Only for the case without rotation, i.e. $\Omega=0$, the smallest eigenvalue is guaranteed to coincide with the (unique) ground state eigenvalue \cite{CCM10}.
\end{remark}

With the above notation, we rephrase assumption \ref{A5} (quasi-isolation of ground states) in terms of an inf-sup stability condition that we can exploit in the error analysis. The inf-sup condition is obtained from the property $\tfrac{ \mbox{\scriptsize d}^{2} }{ \mbox{\scriptsize d} t^{2}} E(\gamma(t)) |_{t=0} > 0$ by explicitly computing the derivative and exploiting the Euler--Lagrange equations. The argument is standard, but for completeness the proof of the following proposition is given in the appendix, Section \ref{appendix-inf-sup-stability}.
\begin{proposition}\label{prop-inf-sup-stability-Wperpu} Assume \ref{A1}-\ref{A5}, let $u \in \mathcal{U}$ be a ground state in the sense of \eqref{definition-groundstate} and $\lambda \in \R_{>0}$ the corresponding ground state eigenvalue given by \eqref{eigen_value_problem_2}. If $\mathcal{I} :H^1_0(\D) \rightarrow H^{-1}(\D)$ denotes the canonical identification $\mathcal{I}v:=(v,\cdot)_{\0}$, then the operator $E''(u) - \lm \mathcal{I}$ is inf-sup stable on $\5 $, namely there exists a constant $\alpha_1 > 0$ such that 
\begin{equation}
\label{inf-sup-stability-Wperpu}
 \inf_{w \in \5} \sup_{v\in \5}\frac{|\langle(E''(u)-\lambda \mathcal{I})v,w\rangle|}{\|v\|_{\1}\|w\|_{\1}} \hspace{2mm} \geqslant \hspace{2mm} \alpha_1.
\end{equation}
In particular, $E^{\prime\prime}(u) - \lambda \mathcal{I}$ has a bounded inverse on $\5$, which follows from standard theory for indefinite problems, cf. \cite{babu70}.
\end{proposition}
Similar assumptions such as the inf-sup condition \eqref{inf-sup-stability-Wperpu} can be frequently found in the literature on the error analysis of finite element approximations for energy minimization problems in quantum physics, cf. \cite{CGHY13,CGHZ11,CHZ11,MadayTurinici2000}.

Using the notation in \eqref{defAcal-u-square}, $E^{\prime\prime}(u)$ in \eqref{inf-sup-stability-Wperpu} can be computed as
\begin{align} 
\label{sd}
\langle E''(u)v,w \rangle =  \langle \Acal_{|u|}v,w \rangle +  2\, \beta \, ( \re (u \overline{v}) \, u,w)_{\0}.  
\end{align}
By selecting $v = \ci u$ in \eqref{sd} and exploiting that $\ci u$ is an eigenfunction of $ \Acal_{|u|}$ with eigenvalue $\lambda$, we have $ \langle E''(u) \ci u,w \rangle = \lambda \, (\ci u,w )_{\0}$ for all $w \in \3$. Hence, $E''(u)-\lambda \mathcal{I}$ is singular on $\mbox{span}\{\ci u\}$ (the directions of complex phase shifts) as previously predicted in Section \ref{uniqueness-section}.

We conclude this section with a natural regularity statement for ground states that we will exploit frequently in the convergence analysis.
\begin{lemma}[$H^{2}$-regularity of ground states]
\label{gs-h2-regularity}
Assume \ref{A1}-\ref{A3}. Let $u$ be a ground state and $\lm >0$ the corresponding ground state eigenvalue as in \eqref{eigen_value_problem_1}. Then, $ u \in C^{0}(\overline{\D}) \cap H^{2}(\D).$
\end{lemma}
\begin{proof}
The proof is straightforward by writing \eqref{eigen_value_problem_1} as 
\begin{equation}
\label{h2_1}
 - \Delta u = \lm u -V u+\Omega \mathcal{L}_{3}u -\beta |u|^2u = :g\,.
 \end{equation}
Our assumptions on $V$ and the Sobolev embedding $H^{1}_{0}(\D) \hookrightarrow L^{6}(\D)$ for $d =2,3$, yield $g \in L^{2}(\D)$. Hence, together with the convexity of $\D$, standard elliptic regularity theory (cf. \cite[Theorem 3.2.2]{cia2002}) for the Poisson problem $-\Delta u= g$ guarantees $ u \in C^{0}(\overline{\D}) \cap H^2(\D)$. 
\end{proof}


\section{Finite element discretization and main results}
\label{fem-main-results}
In the following we investigate the accuracy of discrete ground states to the minimization problem \eqref{definition-groundstate} in finite element spaces. For that, we consider a shape regular family $\{ \mathcal{T}_h\}$  of nested conforming triangulations over the domain $\mathcal{D}$ with mesh size $h >0$. On each mesh $\mathcal{T}_h$, the corresponding $\mathbb{P}^k$-Lagrange finite element space is given by
 $$  V_{h,k} = \{ v \in H^{1}_{0}(\D) \cap C^{0}(\overline{\mathcal{D}}) \vert \,\, v|_{K} \in \mathbb{P}^{k}(K) \hspace{1mm}\mbox{ for all } K \in \mathcal{T}_{h}\},$$
i.e., $V_{h,k}$ is the space of continuous functions in $H_{0}^{1}(\D)$ such that their restriction to any element $K$ of $\mathcal{T}_h$ is a polynomial of degree not greater than $k$. For simplicity we shall mainly denote $V_{h}:= V_{h,k}\,$ as long as the specific value of $k$ is not crucial. If we want to stress the polynomial order, we will use the index \quotes{$h,k$} instead of \quotes{$h$}.

Let us now consider the finite element approximations for ground states of the energy minimization problem \eqref{definition-groundstate}. A ground state in the finite element space $V_h$ is characterized by the minimization problem seeking $u_h \in V_h \cap \mathbb{S}$ such that
\begin{align}
\label{weak_min_problem}
 E(u_{h}) = \inf_{v \in V_h \cap \mathbb{S} } E(v). 
\end{align}
In the following we shall call such minimizers {\it discrete ground states}. As in the continuous case,  discrete ground states $u_{h}$ can be at most unique up to constant phase shifts. 
We denote the set of all discrete ground states by
 \begin{align}
 \label{Uh}
  \mathcal{U}_h := \{ u_{h} \in V_{h} \cap \mathbb{S} \, \vert \, E(u_{h}) = \min_{v \in V_{h} \cap \mathbb{S}} E(v) \} .
  \end{align}
Analogously as in the continuous setting, any minimizer $u_h$ of \eqref{weak_min_problem} can be expressed through the corresponding Euler--Lagrange equations. Hence, there exists $\lambda_h \in \mathbb{R}_{>0}$ such that it holds
\begin{align}
\label{discrete_eigen_value_problem}
(\nablaR u_{h}, \nablaR v)_{\0} + (\VR u_{h} +\beta |u_{h}|^2u_{h},v)_{\0} = \lambda_{h} \, (u_{h},v)_{\0}
\qquad \mbox{for all } v \in V_{h}.
\end{align} 
\begin{remark}[Same phase]
Due to missing uniqueness caused by phase shifts, we can only make a reasonable comparison between an actual ground state $u$ and a discrete ground state $u_h$ if their phases are aligned. In the following we will call $u$ and $u_h$ to be in the \emph{same phase} if their  complex $L^2$-inner product is a value on the positive real axis, i.e. $\int_{\D} u_h \overline{u} \dx \in \R_{\ge 0}$. 
\end{remark}
 To formulate our main results and to keep the proofs compact, we write from now on \quotes{$ a \lesssim b$} to abbreviate \quotes{$a \leq C b $} for a constant $C>0$ that may depend on $\beta$, $\D$,  $\Omega$, $V$ and $u$ but is independent of the mesh size $h$. 
 
We are now prepared to present our main theorems. The first result guarantees that for each discrete ground state $u_h \in \mathcal{U}_h$ (on a sufficiently fine mesh) there is a $H^1$-neighbourhood of $u_h$ in which we will find a unique exact ground state $u \in \mathcal{U}$ that is in the same phase as $u_h$.
Furthermore, such $u_h$ is a quasi-best approximation of $u$ in the $H^1$-norm. 

\begin{theorem} \label{approx_theorem} Assume \ref{A1}-\ref{A4}, then for every sufficiently small $\delta>0$, there is a $h_{\delta}>0$ such that for each $h\le h_{\delta}$ and each $u_h \in  \mathcal{U}_{h}$, there exists a ground state $u \in \mathcal{U}$ that is in the same phase, i.e.,  $\int_{\D}u_h \overline{u} \dx \in \R_{\ge 0}$, and such that 
\begin{align*}
\| u - u_h \|_{\1} \leq \delta \,.
\end{align*}
Furthermore, the error in energy is bounded by
\begin{align}
\label{2.-1}
\left| E(u_h) - E(u) \right| \lesssim \|u_h -u\|_{\1}^2
\end{align}
and distance between the ground state eigenvalue $\lambda$ to $u \in \mathcal{U}$ and the discrete ground state eigenvalue $\lambda_h$ to $u_h \in \mathcal{U}_h$ can be bounded by
\begin{align}
\label{2.0}
\left| \lambda_h - \lambda \right| \lesssim \|u_h -u\|_{\1}^2 +  ( u_h - u ,  u  |u|^2 )_{\0}.
\end{align}
If additionally assumption \ref{A5} is satisfied, then $u$ from the first part of the theorem is \emph{unique} (in the $\delta$-neighborhood); $u_h$ is a $H^1$-quasi-best approximation to $u$
\begin{align}
\label{2.1}
\|u-u_h\|_{H^{1}(\D)} \lesssim \inf_{v_h \in V_h}\|u-v_h\|_{H^{1}(\D)}
\end{align}
and it holds the $L^{2}$-error estimate
\begin{align}
\label{2.2}
\|u-u_h\|_{L^{2}(\D)} \lesssim h \|u-u_h\|_{H^{1}(\D)}. 
\end{align}
\end{theorem}
The proof of Theorem \ref{approx_theorem} is postponed to Section \ref{proof-of-approx-theorem} where we will also state additional complementary results that address the general approximability of all ground states $u \in \mathcal{U}$ (which is important if the ground state density is not unqiue).\\[0.4em]
Theorem \ref{approx_theorem} directly yields optimal convergence rates (in $L^2(\D)$, $H^1(\D)$ and the energy error) for Lagrange spaces of arbitrary polynomial order $\mathbb{P}^k$, provided that the limiting ground state $u$ is sufficiently smooth.  However, it does not contain a corresponding explicit estimate for the error in the ground state eigenvalue. For this we require an additional regularity assumption. The arising optimal convergence rates together with the necessary assumptions are summarized in the following result.
\begin{theorem} \label{order_theorem} Assume \ref{A1}-\ref{A5} and let $(\lambda_{h,k},u_{h,k})$ denote a discrete ground state pair of \eqref{discrete_eigen_value_problem} in $\mathbb{R}_{>0} \times \, \mathcal{U}_{h,k}$ for $k \in \mathbb{N}_{>0}$ and sufficiently small $h$. Let further $(\lm ,u ) \in \mathbb{R}_{>0} \times \, \mathcal{U}$ be the corresponding (uniquely characterized) exact ground state pair with $u_{h,k} \in \tangentspace{\ci u}$ from Theorem \ref{approx_theorem}. If $u \in H^{k+1}(\D)$, then it holds
\begin{align}
\label{pk_1} \|u-u_{h,k}\|_{\0} + h \|u-u_{h,k}\|_{\1} \lesssim h^{k+1}, \hspace{30mm}
\end{align}
and for the error in energy
\begin{align}
\label{pk_2}  | E(u_{h,k}) - E(u) | \lesssim h^{2k}.
\end{align} 
Finally, let $w_{u|u|^2} \in \5$ denote the unique solution to the dual problem
\begin{equation*}
 \langle(E''(u)-\lm \mathcal{I}) w_{u|u|^2} ,\varphi \rangle = \langle \mathcal{I} (u  |u|^2)  ,\varphi \rangle \qquad \mbox{ for all } \varphi \in \5
 \end{equation*}
 (which exists by Proposition \ref{prop-inf-sup-stability-Wperpu}). If $w_{u |u|^2}$ is sufficiently smooth in the sense that $w_{u|u|^2} \in H^{k+1}(\D)$ with the corresponding regularity estimate
 \begin{align}
 \label{adjoint2.2-u3}
 \| w_{u|u|^2} \|_{H^{k+1}(\D)} \lesssim \| \, |u|^3 \|_{H^{k-1}(\D)},
 \end{align}
 then it also holds the following sharp error estimate for the eigenvalue
\begin{align}
\label{pk_3} |\lambda - \lambda_{h,k}| \lesssim h^{2k}.
\end{align} 
\end{theorem}

The proof of Theorem \ref{order_theorem} is presented in Section \ref{proof-of-order-theorem}.

\begin{remark}[Validity of the regularity assumptions in Theorem \ref{order_theorem}]
Note that we always have $u \in H^{2}(\D)$ and $w_{u|u|^2} \in H^{2}(\D)$ with $\| w_{u|u|^2} \|_{H^{2}(\D)} \lesssim \| u \|_{L^6(\D)}^3$, which shows that the regularity assumptions of Theorem \ref{order_theorem} are at least fulfilled with $k=1$ (cf. Lemma \ref{H2-lemma} below). For $k>1$, we would need additional regularity of the potential $V$, as well as a smooth domain $\D$ in order to formally guarantee the required regularity. However, smoothness of $\D$ would be in conflict with our assumption that $\D$ has a polygonal boundary. In practical situations, this is not an issue because the computational domain is selected sufficiently large (estimated by the Thomas-Fermi radius) so that the ground state decays quickly to zero towards the boundary. Hence, potential boundary effects on the regularity are negligible. The same holds for the solution $w_{u|u|^2}$ to the dual problem, which has the quickly decaying function $u|u|^2$ as a source term.

It is also worth to mention that in \cite{CCM10}, additional regularity for rectangular domains could be rigorously proved with a prolongation by reflection argument. The argument does however no longer work in the case of rotation since the $\mathcal{L}_{3}u$-term in \eqref{h2_1} prevents us from writing the equation as a Poisson problem with a source term in $H^1_0(\D)$.
\end{remark}

The rest of the paper is mainly dedicated to the proofs of the two above stated theorems. Before we turn our attention to these proofs, we present two numerical experiments to support the convergence rates \eqref{pk_1}, \eqref{pk_2} and \eqref{pk_3}. Furthermore, we exemplarily computed the lower part of the spectrum of $E''(u)$ to verify that assumption \ref{A5} is fulfilled in our test cases.


\section{Numerical experiments}
\label{exp-second-derivative}

In this section we present numerical experiments for $\mathbb{P}^1$-FEM spaces to validate our theoretical convergence rates for ground states (in the $H^{1}$- and $L^{2}$-norm) and the respective eigenvalue and energy. Moreover, to illustrate the validity of assumption \ref{A5} in our experiments, we compute the spectrum of $E''(u)$ for a ground state $u$ and verify that the ground state eigenvalue $\lm$ is only appearing once and at the bottom of the spectrum.

In the following experiments we consider the energy minimization problem \eqref{definition-groundstate} on the square $\D = [-R,R]^2$, where the particular choice of $R$ will be specified in the two test cases. The problem is discretized with $\mathbb{P}^1$--FEM based on a uniform triangulation of $\D$ consisting of $2 N^2_h$ triangles ($N_h \in \mathbb{N}$) with corresponding mesh size $h = 12/N_h$. The values are specified in the test cases. The resulting discrete minimization problem \eqref{weak_min_problem} is solved with an energy-adaptive Riemannian gradient method suggested in \cite{HeP20}, whose generalization to the GPE with rotation is formulated in \cite[Section 6]{AHP21NumMath}. As a measure of accuracy we will stop the solver whenever the difference between two consecutive energies, i.e, $E(u^{n+1}_h) -E(u^{n}_h)$, falls below the tolerance of $\varepsilon = 10^{-10}$. The starting value for the Riemannian gradient method is chosen, according to \cite[Section 6.1]{BWM05}, as the $L^2$-normalized interpolation of the function
$$ u^{0}(x_1 , x_2) := \frac{\Omega}{\sqrt{\pi}} (x_1 + \ci x_2) \, e^{\frac{-(x_1^2 + x_2^2)}{2}} $$
in $V_h$. In the experiments we chose the harmonic potential $V$ as
$$ V(x_1,x_2):= \frac{1}{2} ( \gamma_{x_1}^2 x_{1}^{2} + \gamma_{x_2}^2 x_{2}^{2}), \qquad \mbox{with } \gamma_{x_1} , \gamma_{x_2} \in \R,$$
where the precise trapping frequencies $\gamma_{x_1}$ and $\gamma_{x_2}$ are specified in the two test cases below.


\begin{figure}[b!]
	\flushleft
	\begin{minipage}[t]{0.5\textwidth}
		\centering
		\includegraphics[scale=0.14]{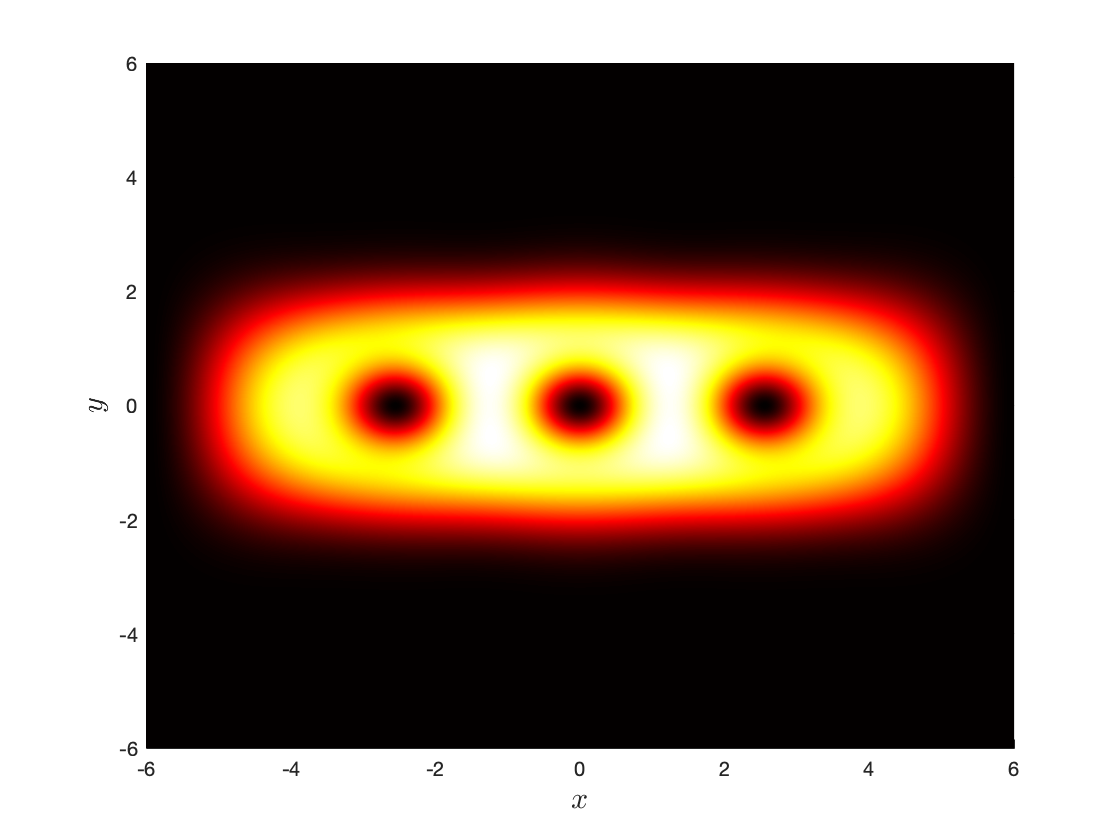}\\
		\vspace{2mm}
		\includegraphics[scale=0.14]{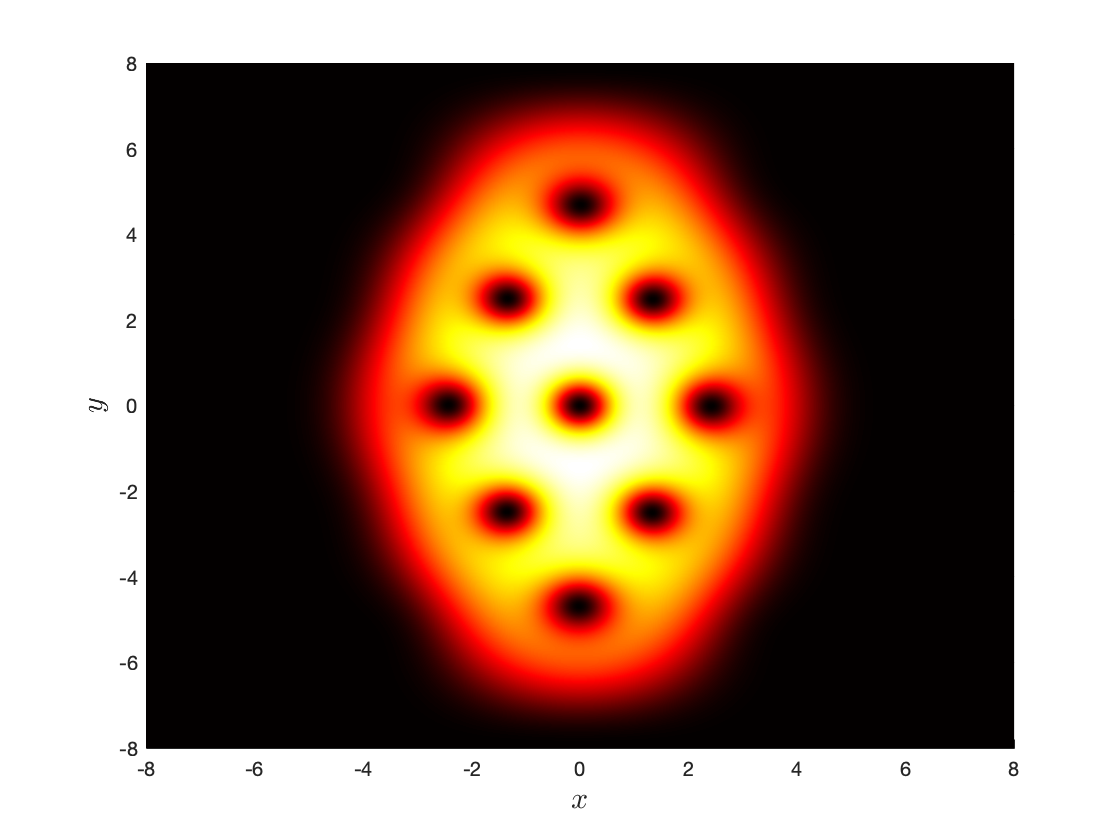}
	\end{minipage}
	\begin{minipage}[t]{0.4\textwidth}
		\centering
		\begin{tabular}{ccc}
			\toprule
			\multicolumn{1}{c}{} & \multicolumn{2}{c}{\textbf{Spectrum of $E''(u)$}}  \\
			\cmidrule(rl){2-3} 
			\textbf{$ i$} & { Model 1} & { Model 2}   \\
			\midrule
			\cellcolor[gray]{.8} 1 & \cellcolor[gray]{.8}4.4488 & \cellcolor[gray]{.8}8.2055  \\
			2 & 4.4609 &  8.2089 \\
			3 & 4.4728 &  8.2111  \\
			4 & 4.4829 &  8.2116  \\
			5 & 4.5013 &  8.2325   \\
			6 & 4.5694 &  8.2379 \\
			7 & 4.6973 &  8.2424  \\
			8 & 4.9089 &  8.2468 \\
			9 & 4.9436 &  8.2479  \\
			10 & 4.9491 & 8.2853  \\
			11 & 4.9500 &  8.3208\\
			12 & 5.0116 &  8.3436  \\
			13 & 5.0415 & 8.3588 \\
			14 & 5.0570 & 8.4840   \\
			15 & 5.1398 & 8.5692   \\
			\bottomrule
		\end{tabular}
	\end{minipage}
	\caption{{\it Surface plot of reference ground state density function $|u_{gs}|^{2}$ in 2D for model 1 (top-left) and model 2 (bottom-left). The table shows the first 15 eigenvalues of second Fr\`echet derivative $E''(u)$ for both the models.} 
	} \label{gs-spectrum}
\end{figure}
\begin{figure}[t!]
	\centering
	\subfloat[$H^{1}$- and $L^{2}$-norm ]{\includegraphics[width=0.5\textwidth]{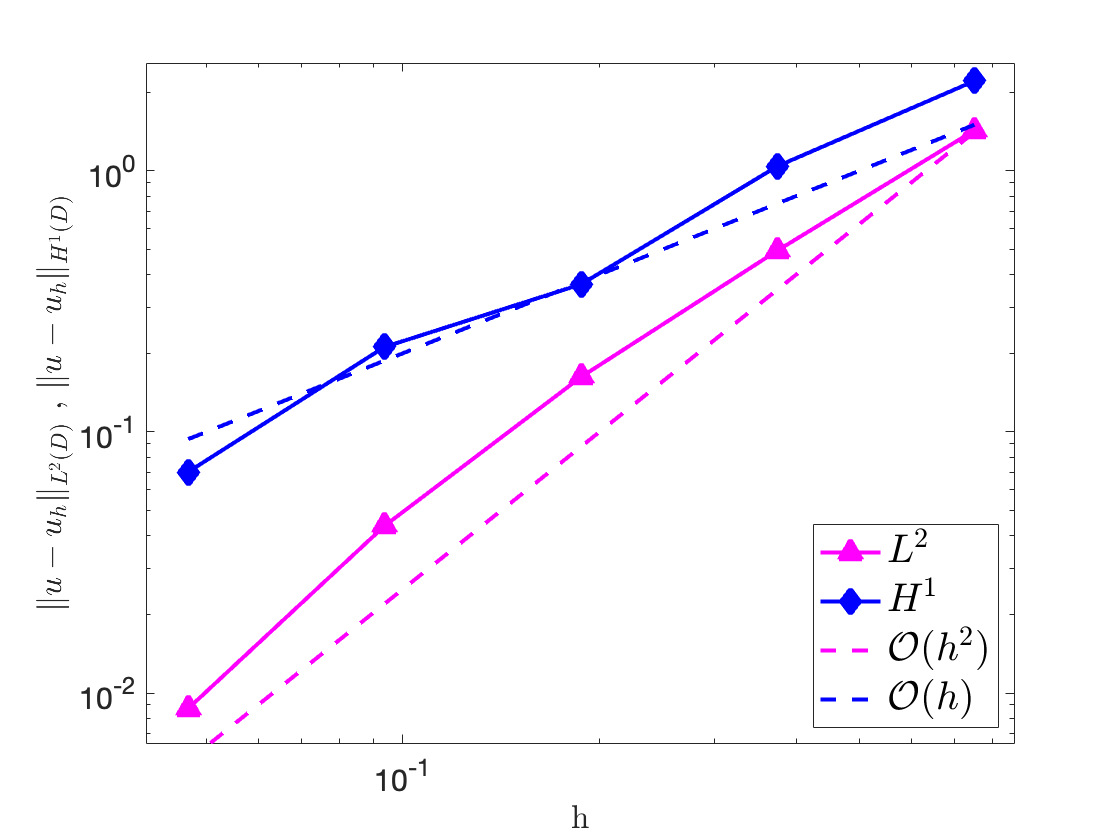}}
	\hfill
	\subfloat[Eigenvalue and energy]{\includegraphics[width=0.5\textwidth]{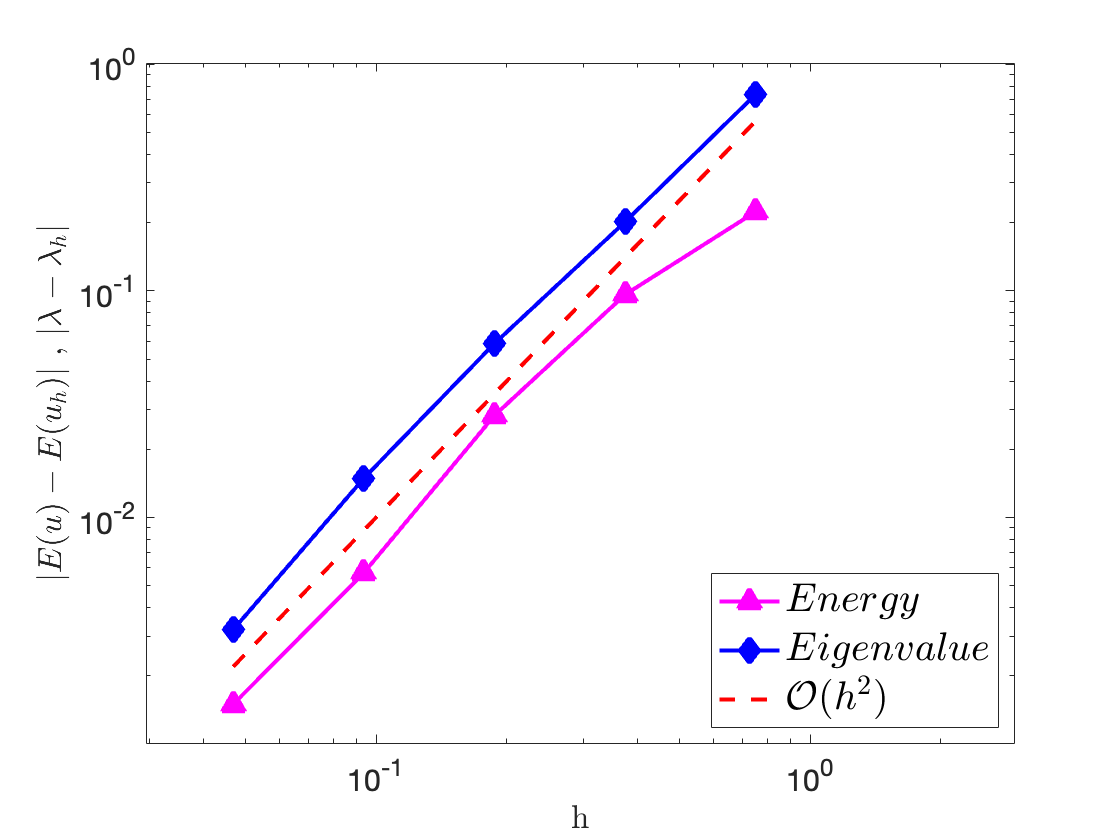}}
	\caption{{\it Errors $\|u-u_{h}\|_{\1}$, $\|u-u_{h}\|_{\0}$,  $| \lm -\lm_h|$ and $|E(u)-E(u_h)|$ for model problem 1.  }}\label{1-rate}
\end{figure}

\subsection{Model problem 1}
{\label{exp-1}}
In the first model problem, we consider $\D=[-6,6]^2$, the angular velocity $\Omega =1.2$ and the repulsion parameter $\beta = 100$. The trapping frequencies are selected as $\gamma_{x_1}= 0.9$ and $\gamma_{x_2} = 1.2$. The discrete ground states are computed on a mesh of decreasing mesh size with values $h= 12 \cdot 2^{-4} \, , 12 \cdot 2^{-5} \, , \, 12 \cdot 2^{-6} \, , \, 12 \cdot 2^{-7} \, , \,12 \cdot 2^{-8} $. The reference ground state $u$ is computed on a fine mesh with size $h = 12 \cdot 2^{-9}$. The problem configuration is such that the ground state $u$
develops quantized vortices. The corresponding density $|u|^2$ of the rotating BEC is plotted (top-left) in Figure \ref{gs-spectrum}. Here we identified the ground state energy with $E(u) \approx 1.6440$ and the corresponding ground state eigenvalue with $\lm \approx 4.4488$. The resulting convergence rates for the ground state (in the $H^{1}$- and in the  $L^{2}$-norm), for the eigenvalue and the energy are plotted in Figure \ref{1-rate} and are in accordance with the predicted rates from Theorem \ref{order_theorem}.

The numerically computed spectrum of the second Fr\`echet derivative $E^{\prime\prime}(u)$, cf. \eqref{sd}, is presented in Figure \ref{gs-spectrum}. As expected, the ground state eigenvalue is a simple eigenvalue of $E^{\prime\prime}(u)$ that only appears at the very bottom of the spectrum. The corresponding eigenfunction is $\ci u$. Note that this shows that assumption \ref{A5} is fulfilled, because it implies (with the Courant--Fischer min-max principle)
\begin{align*}
0<\inf_{v \in \4 } \frac{ \langle(E''(u)-\lambda \mathcal{I})v,v\rangle }{ (v,v)_{\0}} 
\le \inf_{v \in \5 } \frac{ \langle(E''(u)-\lambda \mathcal{I})v,v\rangle }{ (v,v)_{\0}}.
\end{align*}
Hence, the smallest eigenvalue of $E''(u)-\lambda \mathcal{I}$ on $\5$ is positive. Together with the G{\aa}rding inequality, this implies that $E''(u)-\lambda \mathcal{I}$ has a bounded inverse on $\5$ and \ref{A5} is fulfilled.

\begin{remark}[Independence of the spectrum from the phase]
Note that the spectrum of $E''(u)$ does not depend on the phase of $u$, i.e., for all $\theta \in [0,2\pi)$ the eigenvalues of $E''(u)$ and $E''(e^{\ci \theta} u)$ are identical. This observation can be proved with the min-max principle for eigenvalues.
\end{remark}

\subsection{Model problem 2}
{\label{exp-2}}
In the second model problem we set $\D=[-8,8]^2$, the trapping frequencies of the potential to $\gamma_{x_1} = 1.1$ and $\gamma_{x_2}=0.9$; the repulsion parameter to $\beta =400$ and the angular velocity to $\Omega = 1.1$. Here the discrete ground states are computed on a mesh of decreasing mesh size with values $h= 16 \cdot 2^{-4} \, , 16 \cdot 2^{-5} \, , \, 16 \cdot 2^{-6} \, , \, 16 \cdot 2^{-7} \, , \,16 \cdot 2^{-8} $. The reference ground state $u$ is computed on a fine mesh with size $h = 16 \cdot 2^{-9}$. The reference ground state has quantized vortices and we identified the ground state energy with $E(u) \approx 2.9107$ and the corresponding ground state eigenvalue with $\lm \approx 8.2055$. The corresponding density $|u|^2$ (bottom-left ) is plotted in Figure \ref{gs-spectrum}. Again, the convergence rates for the $H^{1}$-error, the $L^{2}$-error, the energy error and the eigenvalue error, as plotted in Figure \ref{2-rate}, are in accordance with Theorem \ref{order_theorem}.

The second Fr\`echet derivative $E^{\prime\prime}(u)$ is presented in Figure \ref{gs-spectrum}. Again, we find that the ground state eigenvalue is the smallest eigenvalue and appears only once in the spectrum, showing (with the same arguments as in Section \ref{exp-1}) that assumption \ref{A5} is fulfilled.

\begin{figure}[t!]
	\centering
	\subfloat[$H^{1}$- and $L^{2}$-norm ]{\includegraphics[width=0.5\textwidth]{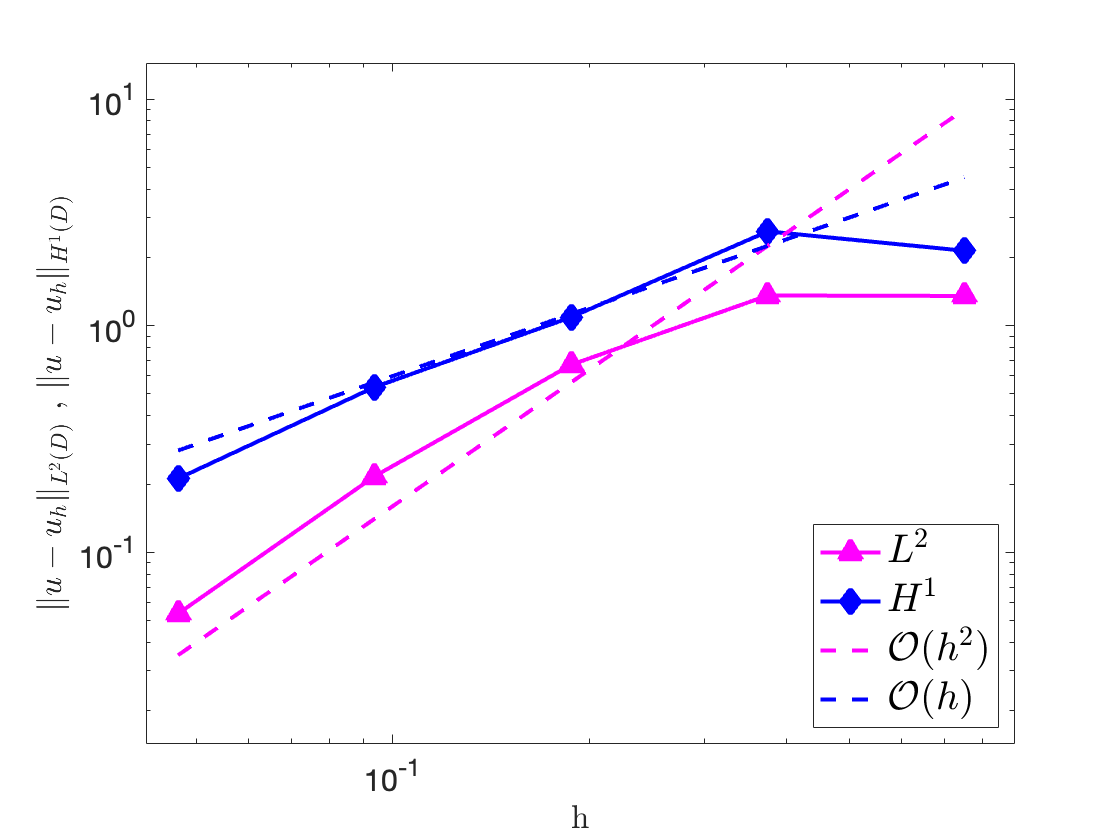}}
	\hfill
	\subfloat[Eigenvalue and energy]{\includegraphics[width=0.5\textwidth]{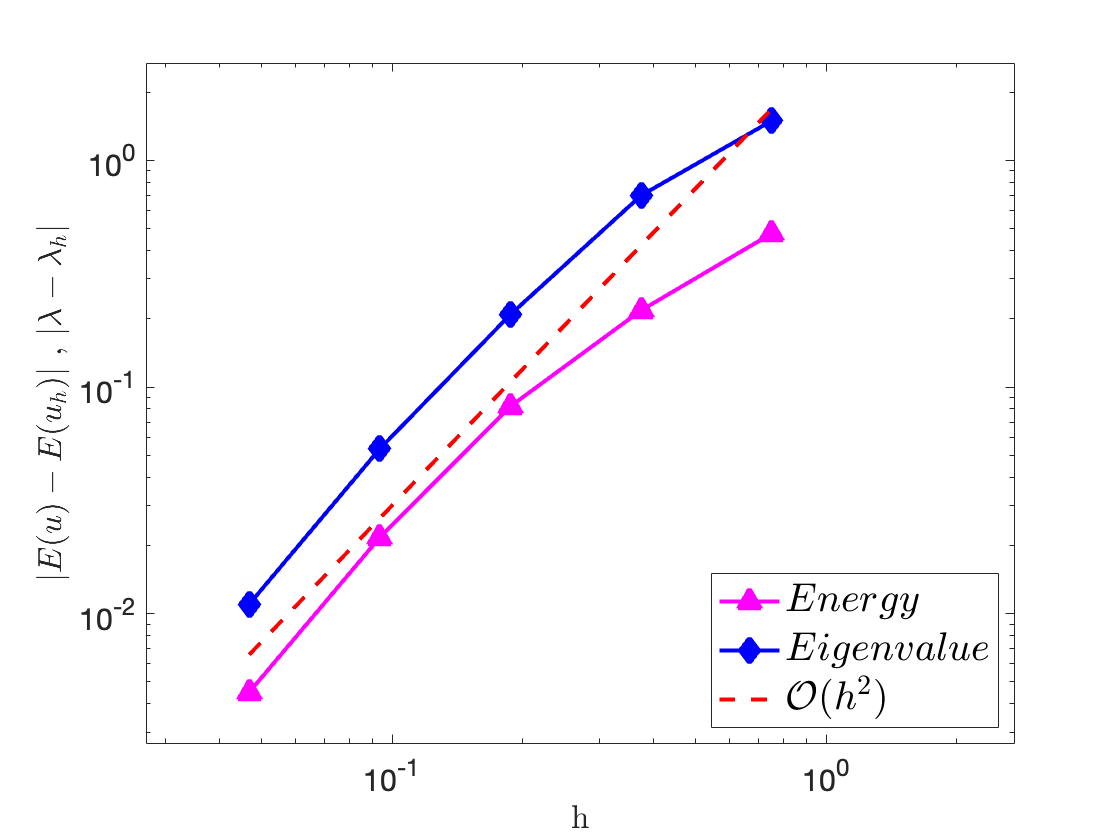}}
	\caption{{\it  Errors $\|u-u_{h}\|_{\1}$, $\|u-u_{h}\|_{\0}$,  $| \lm -\lm_h|$ and $|E(u)-E(u_h)|$ for model problem 2. }}\label{2-rate}
\end{figure}

\section{Proofs of Theorems \ref{approx_theorem} and \ref{order_theorem}} 
\label{proof-of-approx-theorem}

We now turn towards the proofs of our main results which are split into several subsections.

 \subsection{Asymptotic convergence results}  
 \label{basic_analysis}

The aim of this first proof section is to establish the existence of an actual unique quasi-isolated ground state in a sufficiently small $\delta$-neighborhood of each discrete ground state (for $h$ small enough). 

First, note that with assumptions \ref{A1}-\ref{A4} and the norm equivalence \eqref{norm-equivalence-nablaR} together with the fact that $u_{h}$ minimizes the energy of $E$ on the nested spaces $V_h$,  we can conclude that a family of discrete ground states of \eqref{weak_min_problem} must be uniformly bounded in the $H^1$-norm, i.e., there exists a constant $C>0$ such that, for all sufficiently small $h$, it holds
$$ \sup_{ u_{h} \in \mathcal{U}_h} \|u_{h}\|_{\1} \leq \,C.$$
Now we will use the above uniform boundedness along with the lower semi-continuity of the energy functional $E$ in the weak topology of $\mathbb{S}$ to show the $H^1$-convergence of discrete ground states (up to subsequences) to an actual ground state $u \in \mathcal{U}$. The subsequent proposition also proves the first part of Theorem \ref{approx_theorem}.
 \begin{proposition}
\label{convergence_theorem} Assume \ref{A1}-\ref{A4} holds and let $\{(\lm_h , u_h)\} $ be a family of discrete ground state pairs of \eqref{discrete_eigen_value_problem} in $  \mathbb{R}^+ \times \mathcal{U}_h$. Then there exists a subsequence $\{(\lm_{h_j} , u_{h_j})\}_{j \in \mathbb{N}} \subset \{(\lm_h , u_h)\} $ and an exact ground state $(\lm_0 , u_0) \in  \mathbb{R}^+ \times \mathcal{U}$ of  $ \eqref{eigen_value_problem_2}$  such that 
$$ \lim_{j \rightarrow 0} E(u_{h_j}) = \inf_{v \in \mathbb{S}}E(v), $$
$$ \lim_{j \rightarrow 0} \|u_{h_j} - u_0 \|_{\1} = 0 \qquad \mbox{and} \qquad \lim_{j \rightarrow 0} |\lm_{h_j} - \lm_0 | = 0.\hspace{9mm} $$
Furthermore, for each sufficiently small $\delta>0$, there is a $h_{\delta}>0$ such that for each $h\le h_{\delta}$ and each $u_h \in  \mathcal{U}_{h}$, there exists a ground state $u \in \mathcal{U}$ with $\int_{\D}u_h \overline{u} \dx \in \R_{\geq 0}$ and such that 
\begin{align}
\label{convergence_of_each_discrete_gs}
\| u_h - u \|_{\1} \leq \delta \, .
\end{align}
The above $u$ is unique if \ref{A5} is fulfilled. 
\end{proposition}

\begin{proof} The first part of the proof is based on standard compactness arguments, cf. \cite{CGZ10,HeP20}. Let  $ \{u_{h}\}_{h>0} \subset \mathcal{U}_{h}$ be a sequence of discrete ground states of \eqref{weak_min_problem}. Then the Rellich embedding theorem together with the uniform boundedness of $ \{u_{h}\}$ in $H^{1}(\D)$  ensures existence of a weakly convergent subsequence $\{u_{h_j}\}\subset \{u_{h}\}$ and a $u_0 \in H_{0}^{1}(\D)$ such that $u_{h_j} \rightharpoonup  u_0$ in $H_{0}^{1}(\D)$ and  $u_{h_j} \rightarrow u_0$ strongly in $L^{p}(\D)$ for $1 \leq p \leq 4$. 
It remains to prove that $u_0 $ is a ground state of \eqref{definition-groundstate} and that the weak-convergence in $\3$ is actually a strong convergence. To see this, first note that since the family of finite element  spaces $\{V_{h_j}\}_{j \in \mathbb{N}}$ (based on regular mesh refinement) is dense in $H^1_0(\D)$, there exists a minimising sequence $\{ v_{h_j}  \}$ for the energy functional $E$, i.e,  $E(v_{h_j}) \rightarrow \inf\limits_{v \in \mathbb{S}}E(v)$ where $v_{h_j} \in V_{h_j} \cap \mathbb{S}$.  Hence, the lower semi-continuity of $E$ yields
\begin{align}
\label{energylimit}
\inf_{v \in \mathbb{S}}E(v) = \lim_{j \rightarrow 0}E( v_{h_j} )  \geq \liminf_{j \rightarrow 0}E( u_{h_j} ) \geq E(u_0).
\end{align}
Since $\min\limits_{v \in \mathbb{S}} E(v) > E(u_0)$ is not possible, we conclude $\lim\limits_{j \rightarrow 0}E( u_{h_j} )=E(u_0) =\min\limits_{v \in \mathbb{S}}E(v)$. The normalization of the limit, i.e., $\|u_{0}\|_{\0} =1$, follows from the strong $L^2$-convergence of the subsequence $\{u_{h_j}\}\subset \mathbb{S}$ to $u_{0}$. This implies $u_{0}$ is a ground state of \eqref{definition-groundstate}. Consequently, the weak convergence of $\{ u_{h_j} \}$ in $H_{0}^{1}(\D)$ and the strong convergence in $L^{4}(\D)$, together with $\lim\limits_{j \rightarrow 0}E( u_{h_j} )=E(u_0)$, yield
$\lim\limits_{j \rightarrow 0} \| u_{h_j}\|_{\sR} = \| u_0\|_{\sR} $. The norm equivalence of $\| \cdot\|_{\sR}$ and $\|\cdot \|_{H^{1}(\D)}$ together with the Radon-Riesz property imply
$$\lim_{j \rightarrow 0}\|u_{h_j} - u_0 \|_{\1} =0. $$
Also the convergence of the eigenvalue subsequence $\{ \lm_{h_j} \}$ to $\lm_{0}$ follows from the strong convergence of $\{u_{h_j}\}$ to $u_0$ in $H^{1}(\D)$ and $L^{p}(\D)$ ($p=2,4$).

Next, we prove \eqref{convergence_of_each_discrete_gs}. The statement is easily seen with an argument by contradiction. Assume that \eqref{convergence_of_each_discrete_gs} is wrong, then there is a $\delta>0$ and a subsequence $\{ u_{\tilde{h}_j} \} \subset \mathcal{U}_{\tilde{h}_j}$ with $\tilde{h}_j \rightarrow 0$, but $\| u - u_{\tilde{h}_j} \|_{\1} > \delta$ for all ground states $u\in \mathcal{U}$. However, now we could again apply our previous arguments to the sequence $\{ u_{\tilde{h}_j} \}$ to conclude $H^1$-convergence of a subsequence to a ground state. This is a contradiction to  $\| u - u_{\tilde{h}_j} \|_{\1} > \delta$ for all $u\in \mathcal{U}$.

It remains to check that $u \in \mathcal{U}$ can be, without loss of generality, selected in the same phase as $u_h$. For that, let $u \in \mathcal{U}$ be such that \eqref{convergence_of_each_discrete_gs} holds for a given sufficiently small $\delta$. Then
 \begin{align}
 	\label{delta-relation}
 	1\hspace{-1pt}-\hspace{-1pt}\frac{\delta^2}{2} \, \leq \, 
	\re \int_{\D} u_h \overline{u} \dx \leq \, \vert \int_{\D} u_h \overline{u} \dx \, \vert \, \leq 1\hspace{-1pt}+\hspace{-1pt} \delta \, ,
	\enspace
	\mbox{and hence } | \, \im \int_{\D}u_h \overline{u} \dx \, | \, \lesssim \,  \sqrt{\delta} \,  .
 	\end{align}
Now consider the complex number
 $
 \boldsymbol{\alpha} :=  \int_{\D} u_h \overline{u} \hspace{2pt} \dx \, ,
 $ 
 and the respective polar representation as $\boldsymbol{\alpha} = r_h e^{\ci \theta_h} ,$ where $r_h >0$ and $\theta_h \in [0,2 \pi)$. The inequalities in \eqref{delta-relation} imply that $\theta_h$ is in a small neighbourhood of $0$, in particular  $\vert \sin \theta_h \vert \, \lesssim  \sqrt{\delta}$. With this, we consider the phase shift $e^{ \ci \theta_h} u$ of $u$ by the small angle $\theta_h$. The shifted function is still a ground state and obviously in the same phase as $u_h$, i.e. $\int_{\D} u_h \, \overline{e^{ \ci \theta_h} u} \hspace{2pt} \dx \in \R_{\ge 0}$. Furthermore, it holds
\begin{eqnarray*}
 \lefteqn{\| u_h -  e^{\ci \theta_h} u \|_{\1} \le
 \|u_h -u \|_{\1} + |1-e^{\ci \theta_h}| \, \|u\|_{\1} } \\
 & \lesssim  & \|u_h -u \|_{\1} +  \sqrt{2(1-\cos(\theta_h))} \, = \,  \|u_h -u \|_{\1} + 2 \vert \sin \tfrac{\theta_h}{2} \vert  \\
 & \lesssim & \|u_h -u \|_{\1} + 2 \vert \sin \theta_h \vert  \, \, \, \leq \,\,\, \varepsilon(\delta) \, ,
 \end{eqnarray*}
where we used that $\theta_h$ is a small angle and where $\varepsilon(\delta) =  \mathcal{O}(\sqrt{\delta})$. Note that we also exploited the fact that $\|u\|_{\1}$ (for any $u\in \mathcal{U}$) is uniformly bounded by a constant that only depends on the ground state energy and the data parameters.
 Therefore, for  a given $u_h$ we have an actual (quasi-isolated) ground state  $ e^{\ci \theta_h} u \in \mathcal{U} $ in a sufficiently small $\epsilon(\delta)$-neighbourhood of $u_h$ that is also in the same phase as $u_h$, i.e., we have 
 $$\| u_h - e^{\ci \theta_h} u \|_{\1} \leq \varepsilon(\delta) \quad \mbox{ and } \int_{\D} u_h \overline{( e^{\ci \theta_h} \, u)} \hspace{2pt} \dx \in \mathbb{R}_{\ge 0}.$$ 
 This completes the proof as $\varepsilon(\delta) \rightarrow 0$ for $\delta\rightarrow 0$.
\end{proof}
\begin{remark}
	\label{V-perp-u}
 Note that if an actual ground state $u \in \mathcal{U}$ and a discrete ground state $u_h \in \mathcal{U}_h$ are in the same phase, then 
 $$\im \left( \int_{\D} u_h \, \overline{u} \hspace{2pt} \dx \right) = 0 \quad \mbox{ and consequently } \quad \re \left( e^{-\ci \tfrac{\pi}{2}} \int_{\D} u_h \, \overline{u} \hspace{2pt} \dx \right) = 0\, ,$$
 that is, $ u_h \in \tangentspace{\ci u}\, $. In this sense, $ \tangentspace{\ci u} \cap V_h$ becomes a natural space for seeking reasonable approximations of $u$.
\end{remark}

With this we are prepared to turn towards the error estimates.

\subsection{Abstract error bounds for eigenvalue and energy}
\label{subsection:abstracy-bounds}

In this subsection we derive a priori error estimates for the ground state eigenvalue and the ground state energy by bounding it by the $L^2$- and $H^1$-errors for the ground state.

We first present a lemma which gives a useful relation associated with a ground state eigenpair $(\lm,u)$ of \eqref{definition-groundstate} and a discrete ground state pair $(\lm_h , u_h)$ of \eqref{discrete_eigen_value_problem}. Moreover, we obtain an upper bound of the approximation error. The lemma gains interest if we realize that the estimate \eqref{l-lh1-new} helps us to derive the optimal rates for higher order finite element spaces.
\begin{lemma}
\label{lambdah-lambda} Assume \ref{A1}-\ref{A4} holds. 
Let $(\lambda_h , u_h) \in  \mathbb{R}_{>0} \times \mathcal{U}_h $ denote a discrete ground state pair of \eqref{weak_min_problem}  and $(\lambda , u) \in  \mathbb{R}_{>0} \times \mathcal{U}$ 
an arbitrary exact ground state pair. Then it holds
\begin{equation}
\label{l-lh}
\lambda_h - \lambda =\langle (\mathcal{A}_{|u|} -\lm \mathcal{I})(u_{h}-u),u_{h}-u\rangle  + \Ltwo{\beta(|u_h|^2-|u|^2)u_h}{u_h}
\end{equation}
and 
\begin{equation}
\label{l-lh1-new}
|\lambda_h - \lambda|  \lesssim  \|u_h-u\|_{H^1(\D)}^{2} + ( u_h - u , \, u |u|^2 )_{\0}.
\end{equation}
\end{lemma}
Note that \eqref{l-lh1-new} together with $\| u \|_{L^6(\D)} \lesssim \| u \|_{H^1(\D)}$ (for $d\le3$) directly implies
\begin{equation}
\label{l-lh1}
|\lambda_h - \lambda|  \lesssim  \|u_h-u\|_{H^1(\D)}^{2} + \|u_h-u\|_{L^2(\D)},
\end{equation}
which is however not sharp for higher order finite element spaces, whereas \eqref{l-lh1-new} allows to derive optimal rates.

\begin{proof}[Proof of Lemma \ref{lambdah-lambda}]  Recall that the $L^{2}$-inner product only involves the real part. Hence,  it is symmetric and we can write $\langle \Acal_{|u|} (v-u), v-u \rangle  = \langle \Acal_{|u|} v,v\rangle +\langle \Acal_{|u|} u,u \rangle -2\langle \Acal_{|u|} u,v \rangle$. Using the definition of $\langle \Acal_{|u|} v,v\rangle$ and by rearranging the terms we get
$$ (\nablaR  v, \nablaR  v)_{\0}+  (\VR  v,v)_{\0}=  \langle \Acal_{|u|}  (v-u), v-u \rangle  -\langle \Acal_{|u|} u,u \rangle + 2\langle \Acal_{|u|} u,v \rangle -(\beta |u|^2 v,v)_{\0}. $$
Adding $(\beta |v|^2 v,v)_{\0}$ to both the sides of the above expression leads to the identity 
\begin{eqnarray}
\nonumber \lefteqn{\langle \Acal_{|v|} v,v \rangle \,\, = \,\,\langle \Acal_{|u|}  (v-u), v-u \rangle  -\langle \Acal_{|u|} u,u \rangle + 2\langle \Acal_{|u|} u,v \rangle + \beta ( |v|^2-|u|^2 ,|v|^2)_{\0} }\\
\label{l-lh2} & =& \langle (\Acal_{|u|}-\lm \mathcal{I})  (v-u), v-u \rangle +\lambda \langle \mathcal{I}v,v \rangle +\beta ( |v|^2-|u|^2 ,|v|^2)_{\0}\,,\hspace{90pt}
\end{eqnarray}
where the last equality follows from 
$$ \langle \Acal_{|u|} u,u \rangle - 2\langle \Acal_{|u|} u,v \rangle =  \lm \langle \mathcal{I}u,u \rangle -2\lm \langle \mathcal{I}u,v \rangle +\lm \langle \mathcal{I}v,v \rangle -\lm \langle \mathcal{I}v,v \rangle$$
 $$\hspace{18mm}= \lambda \langle \mathcal{I}(v-u),v-u \rangle - \lambda \langle \mathcal{I}v,v \rangle.  $$
Thus, by testing in \eqref{l-lh2} with $v=u_h$ and using the fact that $\|u_h\|_{\0}=1$, we see that \eqref{l-lh} follows immediately.
Now using assumptions \ref{A1}-\ref{A4} and the equivalence of the norms $\| \cdot \|_{\sR}$ and $\| \cdot \|_{\1}$, we bound the first term on the right hand side of \eqref{l-lh} as 
\begin{eqnarray}
\nonumber \lefteqn{|\langle (\mathcal{A}_{|u|} -\lm \mathcal{I})(u_{h}-u),u_{h}-u\rangle |}\\
\nonumber &\leq & \|u_h -u \|^{2}_{\sR} + \|\VR \|_{L^{\infty}(\D)} \|u_h -u\|^{2}_{\0} + \|\, \beta |u|^2 \,\|_{L^{\infty}(\D)}\|u_h -u\|_{\0}^{2}\\
\label{l-lh-new-0}& \lesssim & \|u_h -u \|^{2}_{\1}\,.
\end{eqnarray}
Since $u_h$ is uniformly bounded in $H^1(\D)$ (by the fact that it is a discrete energy minimizer), the second term on the right hand side of \eqref{l-lh} can be estimated as
\begin{eqnarray}
\nonumber \lefteqn{  |\Ltwo{(|u_h|^2-|u|^2)u_h}{u_h} |  } \\
\nonumber &=&  |((u_h -u)^2 , (u_h +u)u_h)_{\0} + ( |u_h -u|^2, (u_h+2u)\overline{u})_{\0} +2\, (u_h -u,u|u|^2)_{\0} | \\
\label{l-lh-new-1} &\lesssim& \| u_h - u\|_{H^1(\D)}^2 +  ( u_h - u , u|u|^2 )_{\0}.
\end{eqnarray}
By combining the estimates \eqref{l-lh}, \eqref{l-lh-new-0} and  \eqref{l-lh-new-1}, we obtain the desired estimate \eqref{l-lh1-new}.
\end{proof}

%
It is shown by Proposition \ref{convergence_theorem} that for a sufficiently fine mesh $h$ and each discrete ground state $u_h$ to \eqref{weak_min_problem}, there exists a (quasi-isolated) ground state of \eqref{definition-groundstate} that is in the same phase and in a small $H^{1}$-neighborhood of $u_h$.  The following lemma (which also proves a part of Theorem \ref{approx_theorem}) estimates the corresponding error in energy and shows that it converges to zero at least with the rate $ \|u_h -u\|_{\1}^2$.
\begin{lemma}
\label{energy_theorem} 
Assume \ref{A1}-\ref{A4}, let $(\lambda_h , u_h) \in  \mathbb{R}_{>0} \times \mathcal{U}_h $ denote a discrete ground state and $(\lambda , u) \in  \mathbb{R}_{>0} \times \mathcal{U}$ an arbitrary exact ground state pair, then it holds
\begin{align}
\label{estimate-energy}
|E(u_h) - E(u) | \lesssim \|u_h -u\|_{\1}^2 \, .
\end{align}
\end{lemma}
\begin{proof} 
Using the definition of the energy $E$ and its first Fr\'echet derivative, we can express the difference of a discrete and the respective actual ground state energy by
\begin{equation}
\label{energy_diff}
E(u_h)-E(u) = \frac{1}{2}\big(\langle \Acal_{|u|}u_h,u_h \rangle - \langle \Acal_{|u|}u,u \rangle \big) +\frac{\beta}{4}\int_{\mathcal{D}}|u_h|^4-|u|^4-2|u|^2(|u_h|^2-|u|^2)\, .
\end{equation}
Noting that $\langle \Acal_{|u|}u_h,u_h \rangle - \langle \Acal_{|u|}u,u \rangle =\langle \Acal_{|u|}(u_h-u),u_h \rangle +\langle \Acal_{|u|}(u_h-u),u \rangle $, the first term on the right hand side of equation \eqref{energy_diff} can be expressed as
 \begin{eqnarray*}
 \lefteqn{ \langle \Acal_{|u|}u_h,u_h \rangle - \langle \Acal_{|u|}u,u \rangle = \langle \Acal_{|u|}(u_h-u),u_h -u\rangle +2 \, \lambda \, (u_h -u,u)_{\0} }\\
&=& \langle \Acal_{|u|}(u_h-u),u_h -u\rangle - \lambda \, (u_h -u,u_h -u)_{\0}
\,\,\,= \,\,\,\langle (\Acal_{|u|}-\lambda \mathcal{I})(u_h -u),u_h -u \rangle,\hspace{150pt}
 \end{eqnarray*}
where the penultimate equality follows from the fact  $\| u_h \|_{\0}=\| u\|_{\0}=1$ since $u,u_h\in \mathbb{S}$. With this, the continuity of $(\Acal_{|u|}-\lambda \mathcal{I})$ implies
 \begin{eqnarray}
 \label{diff.1}
 \left| \langle \Acal_{|u|}u_h,u_h \rangle - \langle \Acal_{|u|}u,u \rangle \right| \lesssim \| u_h -u \|_{\1}^2.
 \end{eqnarray}
 It remains to estimate the second term in \eqref{energy_diff}. For this, we first note that 
\begin{equation*}
 |u_h|^4-|u|^4-2|u|^2(|u_h|^2-|u|^2) = (|u_h|^2 -|u|^2)^2 \, \leq \, |u_h - u|^2 (|u_h| +|u|)^2, 
 \end{equation*}
which we can further estimate with the Cauchy-Schwarz inequality, the uniform boundedness of discrete ground states and the Sobolev embedding $H_{0}^{1}(\D) \hookrightarrow L^{4}(\D)$ for $(d \leq 3)$ by 
\begin{align}
\nonumber |\int_{\mathcal{D}}|u_h|^4-|u|^4-2|u|^2(|u_h|^2-|u|^2) \, | & \lesssim \| \, |u_h - u|^2 \, \|_{L^{2}(\D)}\, \| \, (|u_h| +|u|)^2\, \|_{L^{2}(\D)}\\
\label{diff.2}  & \lesssim  \| \, u_h - u \, \|^{2}_{H^{1}(\D)}.
\end{align}
By putting \eqref{diff.1} and \eqref{diff.2} back in \eqref{energy_diff} we get the desired result. 
\end{proof}


\subsection{Proof of $H^{1}$-quasi optimality}
\label{subsection-H1-estimate}
In order to get explicit estimates for the error, we need to study the $H^1$-best-approximation properties of functions in $\4$ by functions in $V_h \cap \4$, as well as functions in $\5$ by functions in $V_h \cap \5$. The approximation properties in $\4$ are directly required to estimate the $H^1$-error between $u \in \4$ and $u_h \in V_h \cap \4$, whereas the $L^2$-error estimates and the estimates for the eigenvalue error require to study dual problems posed in $\5$.
\begin{lemma}[$H^1$-quasi best-approximations]
\label{projection-property-lemma}
Assume \ref{A1}-\ref{A5}, let $h$ be sufficiently small. Then, for all $v \in \4$, it holds
\begin{align}
\label{projection_property-1} \inf_{v_h \in V_h \cap \4} \|v- v_h\|_{\1} \lesssim  \inf_{v_h \in V_{h}}\|v-v_h\|_{\1}
\end{align}
and for all $w \in \5$ it holds
\begin{align}
\label{projection_property}  \inf_{v_h \in V_h \cap \5} \|w-v_h\|_{\1} \lesssim  \inf_{v_h \in V_{h}}\|w-v_h\|_{\1}.
\end{align}
\end{lemma}
\begin{proof}
Exemplarily we only prove \eqref{projection_property}. The proof of \eqref{projection_property-1} is analogous with obvious modifications. Now, consider the unconstrained $H^1$-projection $P_{H^1,h} : H^1_0(\D) \rightarrow V_h$ given by
$$
( P_{H^1,h} w , v_h )_{\1} =  ( w , v_h )_{\1} \qquad \mbox{for all } v_h \in V_h
$$
and the auxiliary projection $P_{L^2,h}^{\perp} : V_{h}  \rightarrow V_{h} \cap  \5$ that maps a function of $V_h$ to the orthogonal complement of $u$ with respect to the {\it complex} $L^2$-inner product, i.e.,
\begin{align*}
P_{L^2,h}^{\perp} w_h := w_h - \frac{ \int_{\D} w_h \,\, \overline{u} \dx}{ \int_{\D} P_{H^1\hspace{-2pt},h} u \,\, \overline{u} \dx } P_{H^1\hspace{-2pt},h} u 
\qquad \mbox{for } w_h \in V_h.
\end{align*}
We obtain for every $w \in \5$:
\begin{align*}
 \inf_{v_h \in V_{h} \cap \5}\|w-v_h\|_{\1} &\le 
\|w- (P_{L^2,h}^{\perp} \circ P_{H^1\hspace{-2pt},h}) w \|_{\1}  \\
&\le \| w - P_{H^1\hspace{-2pt},h} w \|_{\1}  + \| \frac{ \int_{\D} (w - P_{H^1\hspace{-2pt},h} w) \, \overline{u} \dx }{ \int_{\D} P_{H^1\hspace{-2pt},h} u  \,\, \overline{u} } P_{H^1\hspace{-2pt},h} u \|_{\1} \\
&\le \underbrace{\left( 1 + \frac{\| P_{H^1\hspace{-2pt},h} u \|_{\1} }{1 - \| u - P_{H^1\hspace{-2pt},h}u \|_{\0} } \right)}_{=:C_h}  \| w - P_{H^1\hspace{-2pt},h} w \|_{\1},
\end{align*}
where the constant $C_h$ is uniformly bounded for all sufficiently small $h$ since
\begin{align*}
\| P_{H^1\hspace{-2pt},h} u \|_{\1} \le  \| u \|_{\1} 
\qquad \mbox{and} \qquad  \| u - P_{H^1\hspace{-2pt},h} u \|_{\0} \rightarrow 0 
\mbox{ for } h\rightarrow 0.
\end{align*}
The $H^1$-optimality of $P_{H^1\hspace{-2pt},h}$ on $V_h$ finishes the proof.\\
\end{proof}

In order to derive the desired $H^{1}$-error estimate \eqref{2.1} as stated in Theorem \ref{approx_theorem}, we will make use of the following abstract result which can be e.g. found in \cite[Theorem 4]{PJRJ94}.
\begin{theorem}
\label{unique_discrete_solution}
Let $X$ be a Hilbert space and $X_h \subset X$ a finite dimensional space. Further, let $\Jcal: X \rightarrow X^{*}$ be a $C^{1}$-mapping and $u \in X$ such that 
$$\langle \Jcal(u) ,v \rangle_{X^{*} \times X} = 0 \quad \mbox{ for all } v \in X \, . $$
We assume
\begin{enumerate}
\item[(i)] the bilinear form $ \langle \Jcal'(u) \,\cdot \, , \cdot \, \rangle$ defined on $X \times X$ is symmetric and $\Jcal'(u)$  is an isomorphism from $X$ to $X^{*}$, i.e., there exists a constant $\alpha >0$ such that 
$$ \inf_{v  \in X} \sup_{w  \in X} \frac{\langle \Jcal'(u)w,v \rangle_{X^{*} \times X}}{\| w \|_X \| v \|_{X}} \geq \alpha >0 \, ,$$
\item[(ii)] $\Jcal'(u)$ satisfies the discrete inf-sup condition, i.e., there is a constant $\tilde{\alpha} >0$ such that 
$$ \inf_{v  \in X_h} \sup_{w  \in X_h} \frac{\langle \Jcal'(u)w,v \rangle_{X^{*} \times X}}{\| w\|_X \| v\|_X } \geq \tilde{\alpha} >0 \, ,$$
\item[(iii)] $\Jcal'$ is Lipschitz continuous at $u$, i.e., there exists a constant $\varepsilon_0 >0$ and $L$ such that for all $v \in X$, $\|u-v\|_X \leq \varepsilon_0$, we have 
$$ \|\Jcal'(u) -\Jcal'(v)\|_{X^{*}} \leq L \|u-v\|_{X} \, .$$
\end{enumerate}
Then for each sufficiently small $\delta >0$, there exists a $h_{\delta} >0$ such that for all $h \in (0,h_{\delta}]$ there is unique $u_h \in X_h$ satisfying 
\begin{align*}
 \langle \Jcal(u_h), v_h \rangle_{X^{*} \times X}= 0 \quad \mbox{for all } v_h \in X_h  \quad \mbox{ and } \quad \|u-u_h\|_{X} \leq \delta \, .
\end{align*}
Moreover we have the following error estimate:
\begin{align*}
\|u-u_h\|_{X} \leq \frac{2 \|\Jcal'(u)\|_{X^{*}}}{\alpha} \left( 1+ \frac{\|\Jcal'(u)\|_{X^{*}}}{\tilde{\alpha}} \right) \inf_{v_h \in X_h}\|u-v_h\|_{X} \, .
\end{align*}
\end{theorem}
In order to write the GPEVP in a form that fits into the setting of Theorem \ref{unique_discrete_solution}, we can follow the idea applied in e.g. \cite{CGHY13,CGHZ11,CHZ11} to express the eigenvalue problem \eqref{eigen_value_problem_2} in compact form based on the operator $\Jcal : \R  \times \3 \rightarrow \R  \times \2$ given by 
\begin{align*}
\langle \Jcal (\sigma , v) ,(\tau , w) \rangle_{\mathbb{R} \times \2 , \mathbb{R} \times \3} := \langle  E'(v)-\sigma \Ical v , w \rangle + \frac{\tau}{2}(1- \int_{\D} |v|^2 ) ,  
\end{align*} 
for $ (\sigma , v)$, $(\tau ,w ) \in \R  \times \3$. The space $\mathbb{R} \times \3$ is equipped with the natural norm $\|(\sigma , v)\|_{\mathbb{R} \times \1}:= |\sigma| + \|v\|_{\1}$. Hence, the eigenvalue problem \eqref{eigen_value_problem_2} can be equivalently expressed as: find $(\lm ,u ) \in \R_{>0}  \times \3$ such that 
\begin{align} 
\label{new-eigenvalue-problem}	
 \langle \Jcal (\lm, u) , (\tau , w) \rangle_{\mathbb{R} \times \2 , \mathbb{R} \times \3} =0 \qquad \mbox{ for all } (\tau , w) \in \mathbb{R} \times \3.
 \end{align}
The Fr\`echet derivative of $\Jcal$ at $(\lm , u)$, denoted by $ \Jcal'(\lm ,u) :\R  \times \3 \rightarrow \R  \times \2$, can be computed as
\begin{align}
\label{j-derivative}
\langle \Jcal' (\lm , u) (\sigma , v) , (\tau , w) \rangle_{\mathbb{R} \times \2 , \mathbb{R} \times \3} = \langle  (E''(u)-\lm \Ical)v , w \rangle - \sigma \langle \Ical u,w   \rangle - \tau \langle \Ical u,v \rangle \, ,
\end{align}
for all directions $(\sigma , v) \in \R  \times \3$ and all $(\tau ,w) \in \R \times \3$. It is easy to see that $\mathcal{J}'$ is Lipschitz continuous at $(\lm,u).$ 

In order to apply the abstract result from Theorem \ref{unique_discrete_solution}, we select $X=\R  \times \4$ based on the tangent space in $\ci u$ (noting that $(\lambda,u)\in \R  \times \4$). It remains to establish that $\mathcal{J}'(\lm , u)$ is an isomorphism on $\R  \times \4$ and the corresponding discrete inf-sup stability on $\R \times (V_h \cap \4)$. The first part is covered by the following lemma.
\begin{lemma}
\label{iso-of-J'}
Assume \ref{A1}-\ref{A5} and let $(\lm , u) \in \R_{>0} \times \mathcal{U}$ denote  a ground state pair of \eqref{eigen_value_problem_2}, which hence solves \eqref{new-eigenvalue-problem}. Then $\Jcal'(\lm , u)$ is an isomorphism from $\R  \times \4$ to $\R  \times (\4)^{*}$.
\end{lemma}
\begin{proof}
The lemma is proved if we can show that the equation  
\begin{align}
\label{i1}
\Jcal'(\lm,u) (\sigma , v) = (\gamma ,  f) \, 
\end{align}
admits a unique solution for all $(\gamma ,  f) \in \R  \times (\4)^{*}$. For that, we consider the saddle-point formulation of \eqref{i1}. With the representation of $\Jcal'(\lm,u)$ in \eqref{j-derivative} we obtain:
 \begin{equation}
\begin{split}
\langle  (E''(u)-\lm \Ical)v , w \rangle - \sigma \langle \Ical u , w \rangle &=\langle  f,w \rangle \hspace{22mm} \qquad  \mbox{for all } w \in \4,\\
- \tau \langle \Ical u , v  \rangle&= \langle \gamma \, , \,  \tau \rangle_{\R \times \R} \qquad \hspace{15mm} \mbox{for all } \tau \in \R.
\end{split}\label{saddle-problem}
\end{equation}
The proof is completed by showing that \eqref{saddle-problem} is uniquely solvable. This holds (see \cite[Theorem 1.1]{MR1115205}) iff the variational problem 
\begin{align}
\label{i2}
\langle  (E''(u)-\lm \Ical)v , w \rangle = \langle  f,w \rangle \qquad \mbox{ for all } w \in \text{ker} \,  (\Ical u)_{\vert \4} \, ,
\end{align}
has a unique solution for all $ f \in (\4)^{*}$ and if there exists a constant $\alpha >0$ (independent of $ v, \tau $) such that
\begin{align}
\label{i3}
\alpha \hspace{2mm} \le \hspace{2mm} \inf_{\tau \in \R} \sup_{v \in \4}\frac{-\tau \langle \Ical u ,  v \rangle }{| \tau |\|v\|_{\1}} 
\hspace{2mm} = \hspace{2mm} \inf_{\tau \in \R} \sup_{v \in \4}\frac{\tau \langle \Ical u ,  v \rangle }{| \tau |\|v\|_{\1}}.  
 \end{align}
Since the kernel of  $(\Ical u)_{\vert \4}$ is $\5$, the unique solvability of problem \eqref{i2} follows from Proposition \ref{prop-inf-sup-stability-Wperpu}. Furthermore, we obtain the inf-sup stability \eqref{i3} for arbitrary $\tau \in \R$ with
$$ \sup_{w \in \4}\frac{\tau (u , w)_{\0}}{|\tau| \|w\|_{\1}} \geq \frac{\tau (u, \tau u)_{\0}}{|\tau| \|\tau u\|_{\1}} = \frac{\tau^{2} \|u\|_{\0}^{2}}{|\tau|^{2} \|u\|_{\1}}=\frac{1}{\|u\|_{\1}} = \alpha >0,$$
where we have used the fact that $u \in \mathbb{S}$ and that all ground states $u$ are uniformly bounded in the $H^1$-norm (exploiting energy minimization). Hence, $\Jcal'(\lm , u)$ is an isomorphism from $\R \times \4$ to $\R \times (\4)^{*}$. This completes the proof.
\end{proof}
%
On sufficiently fine meshes, discrete inf-sup stability can be concluded from the continuous inf-sup stability. It is easily seen that the corresponding result is analogous to the proof of Lemma \ref{iso-of-J'} as soon as we have inf-sup stability of $E''(u)-\lambda \mathcal{I}$ on $ V_h \cap \5 $. The proof of the latter statement follows standard arguments and is elaborated in the appendix, see Proposition \ref{discrete-inf-sup-stability}.
With this, we have the following lemma.
\begin{lemma}
\label{dis_infsup_j'}
Assume $h$ to be sufficiently small and that \ref{A1}-\ref{A5} hold. Let $(\lm , u) \in \mathbb{R}_{> 0} \times \mathcal{U}$ be a ground state pair of \eqref{eigen_value_problem_2}, which hence solves \eqref{new-eigenvalue-problem}. Then there exists a $h$-independent constant $\tilde{\alpha}>0$, such that 
$$ \inf_{(\tau_h , w_h)  \in \R \times (V_h \cap \4)} \sup_{(\sigma_h , v_h)  \in \R \times (V_h \cap \4)}\frac{\langle \Jcal '(\lm ,u)(\sigma_h ,v_h), (\tau_h ,w_h) \rangle_{\mathbb{R} \times \2 , \mathbb{R} \times \3}}{\|(\sigma_h , v_h)\|_{\mathbb{R} \times \1} \, \|(\tau_h , w_h)\|_{\mathbb{R} \times \1}  } \geq \tilde{\alpha} \, . $$
\end{lemma}
With Lemma \ref{iso-of-J'} and Lemma \ref{dis_infsup_j'}, we verified that $\mathcal{J}$ fulfills all the assumptions of the abstract Theorem \ref{unique_discrete_solution} with $X=\mathbb{R} \times \4$ and we can straightforwardly prove the following result.
\begin{theorem}
  \label{new-h1-estimate}
 Assume \ref{A1}-\ref{A5} holds and let $u \in \mathcal{U}$ be a (quasi-isolated) ground state. 
 Then for each sufficiently small $\delta >0$ there is a mesh size $h_{\delta}$ such that for all $h \in (0,h_{\delta}]$ there exist a unique discrete stationary state $(\lambda_h , u_h) \in \R_{> 0} \times (V_h \cap \tangentspace{\ci u})$ with $\langle E^{\prime}(u_h),v_h\rangle =\lambda_h \langle \mathcal{I} u_h, v_h \rangle$ for all $v_h \in V_h$ and such that 
 $$\|u-u_h\|_{\1} \leq \delta\, .$$ 
 Moreover, $u_h$ is a $H^{1}$-quasi best approximation to u, i.e.,
 \begin{align}
 \label{new-h1}
 \|u-u_{h}\|_{\1} \lesssim \inf_{v \in V_{h}} \|u-v\|_{\1}.
 \end{align}
 \end{theorem}

\begin{proof}
Let $u \in \mathcal{U}$ be a given ground state with ground state eigenvalue $\lambda$, then it holds $\mathcal{J}(\lambda,u)=0$. Thanks to Lemma \ref{iso-of-J'} and Lemma \ref{dis_infsup_j'}, we see that the assumptions of Theorem \ref{new-h1-estimate} are fulfilled for all sufficiently fines meshes. Hence, we get the existence of a unique discrete stationary state $(\lambda_h , u_h) \in \R_{> 0} \times (V_h \cap \4)$ such that
  \begin{align*}
  | \lm -\lm_h | + \|u-u_{h}\|_{\1} &\lesssim \inf_{(\sigma_h , v_h ) \in \R \times (V_h \cap \4)} (|\lm - \sigma_h | + \|u-v_h\|_{\1}) \\
  &=  \inf_{v_h \in V_h \cap \4}  \|u-v_h\|_{\1}.
  \end{align*}
  The proof is finished by recalling property \eqref{projection_property-1} from Lemma \ref{projection-property-lemma} for $u\in \4$.
\end{proof}  
Note that Theorem \ref{new-h1-estimate} complements Theorem \ref{approx_theorem} in the sense that it states that every exact ground state  $u \in \mathcal{U}$ can be (asymptotically) approximated with quasi-optimal accuracy. On the contrary, Theorem \ref{approx_theorem} does not address arbitrary ground states $u$, but it rather makes predictions about arbitrary discrete ground states $u_h$ and their individual relation to a {\it specific} exact ground state $u$. With this in mind, it remains to formally conclude estimate \eqref{2.1} in Theorem \ref{approx_theorem} from the above setting. 
 \begin{proof}[Proof of estimate \eqref{2.1} in Theorem \ref{approx_theorem}]
Let the mesh size $h$ be small enough and let $u_h \in \mathcal{U}_h$ be a given discrete ground state. From Proposition \ref{convergence_theorem} we know that in a sufficiently small $\delta$-neighborhood of $u_h$ there exists a unique exact ground state $u \in \mathcal{U}$ with $u_h \in \tangentspace{\ci u}$. On the contrary, Theorem \ref{new-h1-estimate} guarantees the existence of unique discrete stationary state $\tilde{u}_h  \in \tangentspace{\ci u}$ in a sufficiently small neighborhood of $u$ for which \eqref{new-h1} holds. Since every discrete ground state is in particular a stationary state, the uniqueness implies $\tilde{u}_h=u_h$ and the $H^1$-quasi best approximation property \eqref{new-h1} is also valid for $u_h$. This proves \eqref{2.1}.
 \end{proof}
 

\subsection{Proof of $L^2$-error estimates}
\label{subsection-L2-estimate}
Recalling that we already proved the estimates \eqref{2.-1} (cf. Lemma \ref{energy_theorem}) and \eqref{2.0} (cf. Lemma \ref{lambdah-lambda}), as well as \eqref{2.1}  (cf. Theorem \ref{new-h1-estimate}), it only remains to check the $L^2$-error estimate \eqref{2.2} to finish the proof of Theorem \ref{approx_theorem}. This will be done in this section.

For brevity, let 
$$ P_h : \5 \rightarrow V_{h} \cap \5$$
denote the $H^1$-projection on $V_{h} \cap \5$ which is defined, for $w\in \5$, by
\begin{align}
\label{definition-Ph}
(  \nabla(w-P_{h}w), \nabla v_h )_{\0} =0 \hspace{3mm}  \qquad \mbox{for all } v_h \in V_{h} \cap \5.
\end{align}
Recalling Lemma \ref{projection-property-lemma}, we know that for all sufficiently small mesh sizes $h$ and all $w\in \5$ it holds
\begin{align}
\label{best-approx-Ph}
\|w-P_{h}w\|_{\1} \lesssim  \inf_{v_h \in V_{h}}\|w-v_h\|_{\1}.
\end{align}
By the properties of the finite element space $V_h$, we have in particular 
\begin{align}
\label{P_property2}
\|w - P_h w \|_{\1} & \lesssim h |w|_{H^2(\D)} \qquad \mbox{for all } w \in \5 \cap H^2(\D).
\end{align}
In the following, we will also make use of an $L^2$-orthogonal decomposition of $H^1_0(\D)$ into
\begin{align}
\label{decom} 
H^1_0(\D) = \mbox{span}\{u\} \oplus \mbox{span}\{\ci u\} \oplus \5 \,.
\end{align}
Here we note that the span refers to the {\it real}-valued span, i.e., $\mbox{span}\{u\} := \{ \alpha u \, | \, \alpha \in \R\}$ and $\mbox{span}\{\ci u\} := \{ \alpha \ci u \, | \, \alpha \in \R\}$. With this, we obtain a decomposition of the error $u_h-u$ as follows.  
\begin{lemma}
\label{error_decomp}
Let $u \in \mathcal{U}$ be a ground state and $u_h \in \mathcal{U}_{h}$ a corresponding discrete ground state such that $\int_{\D}u_h \overline{u} \dx \in \mathbb{R}_{\ge 0}$. Then, there exists $v \in \5$ with $0 \leq \|v\|^2_{\0} \leq 1$ such that
\begin{align}
\label{d2}
 u_h = \sqrt{ 1-\|v\|^2_{\0}}  \,\, u + v \, .
 \end{align}
In particular, with the constant $c(v) := \sqrt{ 1-\|v\|^2_{\0}} - 1$, we can express the error as 
\begin{align}
\label{d3}
u_h - u = c(v)\, u + v \, , \quad \mbox{ where } \quad |c(v)|
\leq \|u_h-u\|^2_{\0}.
\end{align}
\end{lemma}
\begin{proof}
From \eqref{decom}, there exist $c_1 ,c_2 \in \mathbb{R}$ and $v \in \5$ such that
$$ u_h = c_1  u + c_2 \ci u + v \, .$$
Since $u_h, u, v \in \tangentspace{\ci u}$, we have
\begin{align*}
0 = \re \int_{\D}u_h  \, \overline{\ci u} \dx = c_1 \, \re \int_{\D}  u \, \overline{\ci u } \dx + c_2 \,  \re \int_{\D} \, \ci u \, \overline{\ci u }\dx + \re \int_{\D} v \, \overline{\ci u } \dx 
= c_2 + \re \int_{\D} v \, \overline{\ci u } \dx.
\end{align*}
Since $v \in \5$,  we conclude $\int_{\D} v \, \overline{\ci u } \dx=0$ and hence, $c_2 =0$. Next, consider
$$1=(u_h,u_h)_{\0} = (c_1u+v\, , c_1u+v)_{\0} \overset{v \in \tangentspace{\ci u}} = c_1^2 + \|v\|^2_{\0}.$$
Since $c_1^2 \in \mathbb{R}_{\geq 0}$, we conclude 
$$ 0 \, \leq \, c^2_1 ,\, \|v\|^2_{\0} \, \leq \, 1 \qquad \mbox{and} \qquad c_1 = \sqrt{ 1- \|v\|^2_{\0} }.$$
For the error we obtain $ u_h - u = c_1 u +v - u = c(v) u + v \, ,$ where $c(v) = \sqrt{ 1-\|v\|^2_{\0} }-1 $. Since $ | \sqrt{ 1-x}-1| \leq x$ for all $0 \leq x \leq 1$, we conclude $c(v) \le \| v \|^2_{\0}$. Finally, we use $(u_h-u,v)_{\0} = (c(v) u, v)_{\0} + (v,v)_{\0} = \|v\|^2_{\0}$ to verify
$$ \|v\|^2_{\0}= (u_h-u,v)_{\0} \leq \|u_h-u\|_{\0} \|v\|_{\0} \, .$$
Combing the estimates to $c(v) \le \| v \|^2_{\0} \le  \|u_h-u\|_{\0}^2$ yields \eqref{d3} and completes the proof.
\end{proof}

In order to exploit $H^2$-regularity of solutions to the dual problem \eqref{adjoint1-1} below, we require a corresponding result for the operator  $E''(u)-\lm \mathcal{I}$ on $\5$. This is established by the following lemma.
\begin{lemma}[Regularity of dual solutions]
\label{H2-lemma}
Let assumptions \ref{A1}-\ref{A5} hold. Then for any $\phi \in \0$, there exists a unique $w_{\phi} \in H^2(\D) \cap \5$ such that 
\begin{equation}
\label{adjoint1-1}
 \langle(E''(u)-\lm \mathcal{I})w_{\phi},v \rangle = \langle \mathcal{I} \phi ,v\rangle \qquad \mbox{ for all } v \in \5 \, ,
\end{equation}
and the following estimate holds true:
\begin{equation}
\label{h2regularity-general-1}
\|w_{\phi}\|_{H^{2}(\D)} \lesssim \| \phi \|_{\0}.
\end{equation}
\end{lemma}

\begin{proof}
Since $E''(u)-\lambda \mathcal{I}$ fulfills a continuous inf-sup condition \eqref{inf-sup-stability-Wperpu} on $\5$ guaranteed by \ref{A5}, we immediately conclude that \eqref{adjoint1-1} has a unique solution with the stability bound
\begin{align}
\label{regularity-1}
 \|w_{\phi}\|_{\1} \lesssim \|\phi\|_{\0}.
 \end{align}
It remains to prove that $w_{\phi} \in H^2(\D)$ with the regularity estimate \eqref{h2regularity-general-1}. 
For this, we use the expression for the second Fr\'echet derivative given by \eqref{sd} in combination with the definition of the $(\cdot,\cdot)_{\sR}$-inner product in \eqref{R-inner-product}
to rewrite the equation \eqref{adjoint1-1} as
 \begin{align}
 (w_{\phi} , v)_{\sR} & = (f_{\phi},v)_{\0} \, \quad \mbox{for all } v \in \5 \, ,
 \end{align}
where $ f_{\phi}:= \phi + \lambda w_{\phi} - \beta |u|^2 w_{\phi}  - 2\beta \, \re(u \overline{w_{\phi}})u \in \0$. Hence, $\| w_{\phi} \|_{\sR}^2 \le \| f_{\phi} \|_{\0} \| w_{\phi} \|_{\0}$. Since $\| \cdot\|_{\sR}$ is equivalent to the standard $H^1$-norm, we conclude with the definition of $ f_{\phi}$ that
 \begin{align}
 \label{IP-Coercive}
 \| w_{\phi} \|_{\1} \, \lesssim \, \| f_{\phi}\|_{\0} \, \overset{\eqref{regularity-1}}{\lesssim} \,  \|\phi\|_{\0} .
 \end{align} 
Now, let $v \in H^1_0(\D)$. Using the decomposition \eqref{decom}, we can write $v$ uniquely in the form $v = c_1 u + c_2 \ci u + w \, ,$ where $c_1 = (u , v)_{\0} \, , \, c_2 = ( \ci u , v)_{\0}$  and $w \in \5$. This yields
\begin{eqnarray}
\nonumber \lefteqn{ (w_{\phi} , v)_{\sR} \,\, = \,\,
(w_{\phi} ,c_1 u)_{\sR} + (w_{\phi} , c_2 \, \ci u)_{\sR} + (f_{\phi},w)_{\0} } \\
\nonumber & = & \big(w_{\phi} , (u,v)_{\0} u \big)_{\sR} + \big(w_{\phi} , (\ci  u , v)_{\0} \, \ci u \big)_{\sR} + (f_{\phi} ,v )_{\0} \\
\nonumber &   & - \big(f_{\phi} ,(u,v)_{\0} u \big)_{\0} - \big(f_{\phi},(\ci u , v)_{\0} \, \ci u \big)_{\0} \\ 
\nonumber & = &  \big((w_{\phi} ,u)_{\sR} u,v\big)_{\0} + \big((w_{\phi} ,\ci u )_{\sR} \ci u,v \big)_{\0}  + (f_{\phi} ,v )_{\0} \\
\nonumber &   & - \big((f_{\phi},u)_{\0} u , v\big)_{\0} - \big((f_{\phi},\ci u)_{\0} \ci u , v \big)_{\0} \\
 \label{tilde-f} & = & \big( \tilde{f_{\phi}}, v\big)_{\0} \, ,
 \end{eqnarray}
where $ \tilde{f_{\phi}} := (w_{\phi} ,u)_{\sR} u + (w_{\phi} ,\ci u )_{\sR} \ci u + f_{\phi} - (f_{\phi},u)_{\0} u  - (f_{\phi},\ci u)_{\0} \ci u $.  
As $u , f_{\phi} \in \0$, this directly implies $\tilde{f_{\phi}} \in \0$. Furthermore, by the definition of $(\cdot \, , \, \cdot)_{\sR}$ (cf. \eqref{R-inner-product}), we have 
\begin{align}
\nonumber (w_{\phi} , v )_{\sR} & 
= (\nabla w_{\phi}, \nabla v)_{\0} - ( \ci \Omega R^{\top} \nabla w_{\phi},  v)_{\0} + (V w_{\phi} , v)_{\0} \\
\nonumber & = (\nabla w_{\phi}, \nabla v)_{\0} + (g_{\phi} , v)_{\0} \, \qquad \mbox{for } g_{\phi}:=V w_{\phi} -  \ci \Omega R^{\top} \nabla w_{\phi}.
\end{align}
 As $\tilde{f_{\phi}}, g_{\phi} \in \0$ and $v \in H^1_0(\D)$ we conclude that $w_{\phi}  \in \5 \subset H^1_0(\D)$ solves the Poisson equation $-\Delta w_{\phi}= \tilde{f_{\phi}} - g_{\phi}$ on the convex domain $\D$. Hence, standard elliptic regularity theory yields $w_{\phi} \in H^2(\D)$ and $\|w_{\phi}\|_{H^2(\D)} \lesssim \|\tilde{f_{\phi}}\|_{\0} + \|g_{\phi}\|_{\0}$. A direct estimation of these norms together with \eqref{IP-Coercive} gives $ \|\tilde{f_{\phi}}\|_{\0} + \|g_{\phi}\|_{\0} \lesssim  \|\phi\|_{\0}$. This establishes \eqref{h2regularity-general-1} and finishes the proof.
 \end{proof}
Before we can turn towards the proof of the $L^2$-error estimate, we require a final lemma.
\begin{lemma}
\label{fourth} Assume \ref{A1}-\ref{A5}. Let $u \in \mathcal{U}$, $u_h \in \mathcal{U}_h$ and $\delta \in [0,1]$. Then, for any $v \in H_{0}^{1}(\D)$ the following holds
\begin{eqnarray*}
\nonumber \lefteqn{|\langle E''(u+\delta(u_h-u))\, u_h-u\, , v \rangle - \langle E''(u)\, u_h-u \, ,v \rangle| }\\
 & \lesssim & \Big(\|u_h -u\|^2_{\1} + \|u_h -u\|_{\1} \Big) \|u_h -u\|_{\1} \|v\|_{\1}.
\end{eqnarray*}
\end{lemma}

\begin{proof} By using the definition of the second Fr\'echet derivative, we can express 
\begin{eqnarray}
\nonumber \lefteqn{ \langle E''(u+\delta(u_h-u))\, u_h-u \, , v \rangle - \langle E''(u)\, u_h-u \, , v\rangle}\\
\nonumber &=& \Big( ( |u+\delta(u_h-u)|^2 -|u|^2)(u_h-u) , v\Big)_{\0} - \Big( \re \big( u(\overline{u_h-u})\big)u , v \Big)_{\0}\\
\label{e3.1}& &  + \Big( \re\big( (u+\delta(u_h-u))(\overline{u_h-u})\big) (u+\delta(u_h-u)) , v \Big)_{\0}. 
\end{eqnarray}
We can now estimate the right-hand side of\eqref{e3.1} by using H\"older's inequality and the Sobolev embedding  $H_{0}^{1}(\D)$ $ \hookrightarrow$ $L^{3}(\D),L^{6}(\D)$ for $(d \leq 3)$ in the following way. To do so, we consider the first term of \eqref{e3.1} and obtain
\begin{eqnarray}
\nonumber \lefteqn{\Big( ( |u+\delta(u_h-u)|^2 -|u|^2)(u_h-u) , v \Big)_{\0} }\\
\nonumber &=&  \delta^2 \Big(|u_h-u|^2(u_h-u) , v \Big)_{\0} +   \delta \Big( \big(u (\overline{u_h-u}) + \overline{u} (u_h-u) \big) \big(u_h-u \big), v \Big)_{\0} \\
\nonumber & \lesssim & \Big(\||u_h-u|^2\|_{L^3(\D)} \|u_h-u\|_{L^6(\D)}  + \|u\|_{L^{\infty}(\D)} \|u_h-u\|_{L^3(\D)} \|u_h-u\|_{L^6(\D)} \Big) \| v\|_{\0} \\
\nonumber & \lesssim & \Big( \|u_h-u\|^2_{L^6(\D)}  + \|u_h-u\|_{L^3(\D)} \Big) \|u_h-u\|_{L^6(\D)} \| v\|_{\0} \\
\nonumber & \lesssim & \Big( \|u_h-u\|^2_{\1}  + \|u_h-u\|_{\1} \Big) \|u_h-u\|_{\1} \| v\|_{\0} \, .
\end{eqnarray}
We can now proceed analogously to bound the remaining terms in \eqref{e3.1} as
\begin{eqnarray}
\nonumber \lefteqn{\Big( \re\big( (u+\delta(u_h-u))(\overline{u_h-u})\big) (u+\delta(u_h-u)) , v \Big)_{\0}- \Big( \re \big( u(\overline{u_h-u})\big)u , v \Big)_{\0} }\\
\nonumber &=& \Big( \re\big( u(\overline{u_h-u})\big) \delta(u_h-u) , v \Big)_{\0} + \Big( \re\big( \delta |u_h-u|^2\big) u , v \Big)_{\0} \\
\nonumber & & + \Big( \re\big( \delta |u_h-u|^2\big) \delta(u_h-u) , v \Big)_{\0}\\
\nonumber & \lesssim & \Big( \|u\|_{L^{\infty}(\D)} \|u_h-u\|_{L^3(\D)} \|u_h-u\|_{L^6(\D)} + \| \, |u_h-u|^2\|_{L^3(\D)} \|u_h-u\|_{L^6(\D)} \Big) \| v\|_{\0}\\
\nonumber  & \lesssim & \Big(\|u_h -u\|^2_{\1} + \|u_h -u\|_{\1} \Big) \|u_h -u\|_{\1} \|v\|_{\1} \,.
\end{eqnarray}
This completes the proof.
\end{proof}
\begin{proof}[Proof of estimate \eqref{2.2} in Theorem \ref{approx_theorem}]
Noting that the rest of Theorem \ref{approx_theorem} is already proved, we assume that $h$ is sufficiently small and we let $u \in \mathcal{U}$ and $u_h \in \mathcal{U}_h \cap \4$ denote a ground state pair such that $u_h$ is a $H^1$-quasi-best approximation of $u$ in the sense of \eqref{2.1}. Using Lemma \ref{error_decomp} there exists $v \in \5$ and $c(v) \in \mathbb{R}$ such that
\begin{align}
\label{error-decomp-1}
 u_h - u = c(v)\, u + v \, , \quad \mbox{ where } \quad |c(v)| \leq \|u_h-u\|^2_{\0}.
 \end{align}
By Lemma \ref{H2-lemma} there exists a unique solution $w_{u_h-u} \in \5 \cap H^2(\D)$ such that 
\begin{equation}
\label{adjoint1}
 \langle(E''(u)-\lm \mathcal{I})w_{u_h-u},v \rangle = \langle \mathcal{I}(u_h-u),v\rangle \qquad \mbox{ for all } v \in \5
\end{equation}
and
\begin{equation}
\label{regularity}
\|w_{u_h-u}\|_{H^{2}(\D)} \lesssim \|u_h-u\|_{\0}.
\end{equation}
We obtain
\begin{eqnarray}
\label{step1-L2-est} \lefteqn{\|u_h-u\|^2_{\0} 
 \,\,\,\overset{\eqref{error-decomp-1}}{=}\,\,\, (u_h-u,c(v)\, u)_{\0} + (u_h-u,v)_{\0} } \\
\nonumber &\overset{\eqref{adjoint1}}{=}& (u_h-u,c(v)\, u)_{\0} + \langle(E''(u)-\lm \mathcal{I})w_{u_h-u},v \rangle \\
\nonumber &\overset{\eqref{error-decomp-1}}{=}& (u_h-u,c(v)\, u)_{\0} - \langle(E''(u)-\lm \mathcal{I})w_{u_h-u},c(v) u \rangle + \langle(E''(u)-\lm \mathcal{I})w_{u_h-u},u_h-u \rangle \\
\nonumber &=& (u_h-u,c(v)\, u)_{\0} - \langle(E''(u)-\lm \mathcal{I})w_{u_h-u},c(v) u \rangle  \\
\nonumber & &  + \langle(E''(u)-\lm \mathcal{I}) (w_{u_h-u} - P_{h}w_{u_h-u}) ,u_h-u \rangle + \langle(E''(u)-\lm \mathcal{I})P_{h}w_{u_h-u},u_h-u \rangle.
\end{eqnarray}
Next, we note that from the continuous and discrete eigenvalue problem, we have
\begin{eqnarray}
\label{step2-L2-est}
\lefteqn{ \langle(E''(u)-\lm \mathcal{I}) P_{h}w_{u_h-u},u_h-u \rangle} \\ 
\nonumber&=& \langle E''(u)\, P_{h}w_{u_h-u},u_h-u \rangle - \langle E'(u_h), P_{h}w_{u_h-u} \rangle + \langle E'(u), P_{h}w_{u_h-u} \rangle +  (\lambda_h -\lambda) \langle \Ical u_h \, , P_{h}w_{u_h-u} \rangle \,.
\end{eqnarray}
Here, the last term in \eqref{step2-L2-est} can be expressed, using $w_{u_h-u} \in \5 $, as
$$
 \langle \Ical u_h \, , P_{h}w_{u_h-u} \rangle  = \langle \Ical (u_h - u) \,,  w_{u_h-u} \rangle + \langle \Ical u_h \, , P_{h}w_{u_h-u} - w_{u_h-u} \rangle,
 $$
and the middle two terms in \eqref{step2-L2-est} can be written using Taylor expansion as
\begin{align*}
\langle E'(u_h), P_{h}w_{u_h-u} \rangle - \langle E'(u), P_{h}w_{u_h-u} \rangle = \langle E''(u+\delta(u_h-u))\, P_{h}w_{u_h-u},u_h-u \rangle
\end{align*}
for some $\delta \in [0,1]$. Inserting the last two identities into \eqref{step2-L2-est} and afterwards the result into \eqref{step1-L2-est} yields
\begin{eqnarray}
\nonumber \lefteqn{ \|u_h-u\|^2_{\0} \,\,\,=\,\,\, 
(u_h-u,c(v)\, u)_{\0} - \langle(E''(u)-\lm \mathcal{I})w_{u_h-u},c(v) u \rangle }\\
\nonumber & & + \langle(E''(u)-\lm \mathcal{I}) (w_{u_h-u} - P_{h}w_{u_h-u}) ,u_h-u \rangle \\
\nonumber & & - \Big( \langle E''(u+\delta(u_h-u))\, P_{h}w_{u_h-u},u_h-u \rangle - \langle E''(u)\, P_{h}w_{u_h-u},u_h-u \rangle \Big) \\
\label{estimate} & & +  (\lambda_h -\lambda) \Big(\langle \Ical (u_h - u) \, , w_{u_h-u} \rangle + \langle \Ical u_h \, , P_{h}w_{u_h-u} - w_{u_h-u} \rangle \Big).
\end{eqnarray}
Next, we will estimate the terms appearing on the right hand side of the above equality \eqref{estimate} separately. The first term is bounded with estimate \eqref{d3} as
\begin{align}
\label{e1}
 (u_h-u,c(v)\, u)_{\0} \lesssim \|u_h-u\|_{\0}^3 .
\end{align} 
To estimate the second term in \eqref{estimate}, we first use the continuity property of bilinear form $\langle(E''(u)-\lm \mathcal{I})\, \cdot \,, \, \cdot \, \rangle$ and then the stability estimate $\|w_{u_h-u}\|_{\1} \lesssim \|u_h -u \|_{\0}$ to obtain
\begin{align}
\label{e2}
\langle(E''(u)-\lm \mathcal{I})w_{u_h-u},c(v) u \rangle \lesssim \|w_{u_h-u}\|_{\1} \|c(v) u \|_{\1} \overset{\eqref{d3}}\lesssim \|u_h -u \|^{3}_{\0}\, .
\end{align}
For the third term in \eqref{estimate} we exploit  $w_{u_h-u} \in H^2(\D)$ and the corresponding regularity estimate \eqref{regularity} to conclude
\begin{eqnarray}
\nonumber \lefteqn{\langle(E''(u)-\lm \mathcal{I}) (w_{u_h-u} - P_{h}w_{u_h-u}) ,u_h-u \rangle \,\,\,
 \lesssim \,\,\, \|w_{u_h-u} - P_{h}w_{u_h-u} \|_{\1} \|u_h-u\|_{\1} } \\
\nonumber & \overset{\eqref{P_property2}}{\lesssim} & h \|w_{u_h-u}\|_{H^2(\D)} \|u_h-u\|_{\1} \\
\label{e3} &\overset{\eqref{regularity}}{\lesssim} & h \|u_h-u\|_{\0} \|u_h-u\|_{\1} \, .\hspace{250pt}
\end{eqnarray}
By using Lemma \ref{fourth} with $v=P_h w_{u_h-u}$, we can estimate the fourth term in \eqref{estimate} as
\begin{eqnarray}
\nonumber \lefteqn{|\langle E''(u+\delta(u_h-u))\, u_h-u\, , v \rangle - \langle E''(u)\, u_h-u \, , P_h w_{u_h-u} \rangle| }\\
\nonumber & \lesssim & \Big(\|u_h -u\|^2_{\1} + \|u_h -u\|_{\1} \Big) \|u_h -u\|_{\1} \| P_h w_{u_h-u} \|_{\1} \\
\label{e4}  & \lesssim & \Big(\|u_h -u\|^2_{\1} + \|u_h -u\|_{\1} \Big) \|u_h -u\|_{\1} \| u_h-u \|_{\0} \, ,
\end{eqnarray}
where the last step follows from the $H^1$-stability of $P_h$. Finally, the last term in \eqref{estimate} is again estimated using the approximation property of projection $P_h$ together with \eqref{regularity}:
\begin{eqnarray}
\nonumber \lefteqn{\Big(\langle \Ical (u_h - u) \, , w_{u_h-u} \rangle + \langle \Ical u_h \, , P_{h}w_{u_h-u} - w_{u_h-u} \rangle \Big)  }\\
\label{e5} & \overset{\eqref{regularity}}{\lesssim} & \|u_h-u\|_{\1} \|u_h-u\|_{\0} + h \|u_h-u\|_{\0} \, .
\end{eqnarray} 
Substituting the estimates \eqref{e1}--\eqref{e5} into \eqref{estimate} and using the (suboptimal) eigenvalue estimate \eqref{l-lh1}, i.e., $|\lambda_h - \lambda|  \lesssim  \|u_h-u\|_{H^1(\D)}^{2} +  \|u_h-u\|_{L^2(\D)}$, we get 
\begin{eqnarray}
\label{step3-L2-est} \lefteqn{\|u_h-u\|^2_{\0} \,\,\, \lesssim \,\,\, \|u_h-u\|_{\0}^3 + h \|u_h-u\|_{\0} \|u_h-u\|_{\1} } \\
\nonumber& & \hspace{10pt} + \Big(\|u_h -u\|^2_{\1} + \|u_h -u\|_{\1} \Big) \|u_h -u\|_{\1} \|u_h -u\|_{\0} \\
\nonumber& & \hspace{10pt} + \Big( \|u_h-u\|_{H^1(\D)}^{2} +  \|u_h-u\|_{L^2(\D)} \Big) \Big( \|u_h-u\|_{\1} \|u_h-u\|_{\0} + h \|u_h-u\|_{\0} \Big).
\end{eqnarray}
Finally, we can use the quasi-best-approximation property \eqref{2.1} together with the $H^2$-regularity from Lemma \ref{gs-h2-regularity} to conclude $\|u_h-u\|_{\1} \lesssim h\,|u|_{H^2(\D)}$. Plugging this into \eqref{step3-L2-est} and dividing by $\|u_h-u\|_{\0}$ finally yields
\begin{align*}
\|u_h-u\|_{\0} \lesssim h \|u_h-u\|_{\1}.
\end{align*}
\end{proof}

 \subsection{Proof of optimal convergence rates for the eigenvalue}  
  \label{proof-of-order-theorem}
Theorem \ref{approx_theorem} guarantees that for each discrete ground state pair $(\lm_{h,k} , u_{h,k}) \in \mathbb{R}_{>0} \times \mathcal{U}_{h,k}$ to \eqref{discrete_eigen_value_problem} (on a sufficiently fine mesh), there is a unique quasi-isolated ground state pair $(\lm ,u) \in \mathbb{R}_{>0} \times \mathcal{U}$ to \eqref{eigen_value_problem_2} such that $u_{h,k} \in \tangentspace{\ci u}$, as well as
\begin{align*}
\|u-u_{h,k}\|_{H^{1}(\D)} \lesssim \inf_{v_h \in V_{h,k}}\|u-v_h\|_{H^{1}(\D)}   \qquad \mbox{and} \qquad \|u- u_{h,k} \|_{L^{2}(\D)} \lesssim h \|u- u_{h,k}\|_{H^{1}(\D)} .
\end{align*}
If the ground state $u$ has higher order regularity, i.e., $u \in H^{k+1}(\mathcal{D})$, then we directly get the following optimal order estimates for the $L^2$- and $H^1$-error (cf. \cite[Theorem 3.2.2]{cia2002}):
\begin{align}
\label{id1}
\|u-u_{h,k}\|_{H^1(\D)} \lesssim  h^{k} |u|_{H^{k+1}(\D)} \hspace{5mm} \text{and} \hspace{5mm} \|u-u_{h,k}\|_{L^2(\D)} \lesssim h^{k+1}|u|_{H^{k+1}(\D)}.
\end{align}
On top of that, we obtain from \eqref{estimate-energy} an optimal order estimate for the energy with 
\begin{align}
\label{id2}
\left| E(u) - E(u_{h,k}) \right| \lesssim \| u_{h,k}  - u\|_{\1}^2 \lesssim h^{2k} |u|_{H^{k+1}(\D)}^2.
\end{align}
A sharp estimate for the eigenvalue is less straightforward. Here we recall from \eqref{l-lh1-new} in Lemma \ref{lambdah-lambda}  that  
\begin{align}
\label{id3}
|\lambda - \lambda_{h,k}|  \lesssim  \| u - u_{h,k} \|_{H^1(\D)}^{2} +  ( u_{h,k} - u ,  u  |u|^2 )_{\0},
\end{align}
which we could crudely estimate (cf. \eqref{l-lh1}) by
\begin{equation}
\label{suboptimal-estimate}
|\lambda - \lambda_{h,k}|  \lesssim  \| u - u_{h,k} \|_{H^1(\D)}^{2} + \| u - u_{h,k} \|_{L^2(\D)}
\lesssim h^{2k} |u|_{H^{k+1}(\D)}^2 + h^{k+1} |u|_{H^{k+1}(\D)}.
 \end{equation}
However, as the optimal rate claimed in Theorem \ref{order_theorem} is of order $\mathcal{O}(h^{2k})$, the above estimate for the ground state eigenvalue is only optimal for $k=1$, i.e., for  $\mathbb{P}^1$-FE approximations. In order to improve the result for arbitrary polynomial orders, we require additional regularity for the solution to a suitable dual problem with source term $u |u|^2$. In the following, we will elaborate the argument and prove \eqref{pk_3} which then concludes Theorem \ref{order_theorem}.
\begin{proof}[Proof of Theorem \ref{order_theorem}]
As described above, it only remains to check the eigenvalue estimate \eqref{pk_3}. We start from \eqref{id3} to observe that we only need to prove that $( u_{h,k} - u ,  u  |u|^2 )_{\0} \lesssim h^{2k}$ to obtain the desired estimate of optimal order for $|\lambda - \lambda_{h,k}|$. Following a similar strategy as for the $L^2$-estimate in  Theorem \ref{approx_theorem}, we consider the unique solution $w_{u|u|^2} \in \5$ to the dual problem
\begin{equation}
\label{adjoint2.1-u3}
 \langle(E''(u)-\lm \mathcal{I}) w_{u|u|^2} ,\varphi \rangle = \langle \mathcal{I} (u  |u|^2)  ,\varphi \rangle \qquad \mbox{ for all } \varphi \in \5.
 \end{equation}
 Recall here our regularity assumption that $w_{u|u|^2} \in H^{k+1}(\D) \cap H^1_0(\D)$ with $\|w_{u|u|^2}\|_{H^{k+1}(\D)} \lesssim \| \, |u|^3 \|_{H^{k-1}(\D)}$. As in our previous estimates, it holds
\begin{eqnarray}
\nonumber \lefteqn{( u_{h,k} - u  ,  u|u|^2 )_{\0}  \,\,\, = \,\,\, (u|u|^2,c(v)\, u)_{\0} - \langle(E''(u)-\lm \mathcal{I})w_{u|u|^2},c(v) u \rangle }\\
\nonumber & & + \,\,\,\langle(E''(u)-\lm \mathcal{I}) (w_{u|u|^2} - P_{h}w_{u|u|^2}) ,u_{h,k}-u \rangle \\
\nonumber & & +\,\,\, \langle E''(u)\, P_{h}w_{u|u|^2},u_{h,k}-u \rangle - \langle E''(u+\delta(u_{h,k}-u))\, P_{h}w_{u|u|^2},u_{h,k}-u \rangle \\
\label{analyze1.1-new} & & + \,\,\,  (\lambda_h -\lambda) \Big(\langle \Ical (u_{h,k} - u) \, , w_{u|u|^2} \rangle + \langle \Ical u_{h,k} \, , P_{h}w_{u|u|^2} - w_{u|u|^2} \rangle \Big).
\end{eqnarray}
 We can now proceed analogously as in the proof of Theorem \ref{approx_theorem} to estimate the terms in \eqref{analyze1.1-new}. For the first term we apply Lemma \ref{error_decomp} to obtain
\begin{align*}
 (u|u|^2,c(v)\, u)_{\0} \lesssim \|u_{h,k}-u\|_{\0}^2 \overset{\eqref{id1}}{\lesssim} h^{2k+2}  
\end{align*} 
In the same way, the second term in \eqref{analyze1.1-new}
is straightforwardly bounded by
\begin{align*}
 \langle(E''(u)-\lm \mathcal{I})w_{u|u|^2},c(v) u \rangle  \overset{\eqref{h2regularity-general-1}}{\lesssim} \| \, |u|^3 \|_{\0} \|u_{h,k}-u\|^2_{\0} \|u\|_{\1}  \lesssim  h^{2k+2} .
\end{align*}
The approximation properties \eqref{best-approx-Ph} of $P_h$ yield for the third term in \eqref{analyze1.1-new} that
\begin{align*}
 \langle(E''(u)-\lm \mathcal{I}) (w_{u|u|^2} - P_{h}w_{u|u|^2}) ,u_{h,k}-u \rangle \,\, \lesssim \,\, h^{k} \|w_{u|u|^2}\|_{H^{k+1}(\D)} \|u_{h,k}-u\|_{\1} \overset{\eqref{id1},\eqref{adjoint2.2-u3}}{\lesssim}  h^{2k} \, .
\end{align*}
For term four we get with Lemma \ref{fourth}
\begin{eqnarray}
\nonumber \lefteqn{|\langle E''(u+\delta(u_{h,k}-u))\, (u_{h,k}-u)\, , P_h w_{u|u|^2} \rangle - \langle E''(u)\, (u_{h,k}-u) \, , P_h w_{u|u|^2} \rangle| }\\
 \nonumber & \lesssim & \Big(\|u_{h,k} -u\|^2_{\1} + \|u_{h,k} -u\|_{\1} \Big) \|u_{h,k} -u\|_{\1} \| u|u|^2 \|_{L^{2}(\D)} \,\,\lesssim \,\, h^{2k} \, .
\end{eqnarray}
For the last term in \eqref{analyze1.1-new} we obtain 
\begin{eqnarray}
\nonumber \lefteqn{(\lambda_h -\lambda) \Big(\langle \Ical (u_{h,k} - u) \, , w_{u|u|^2} \rangle + \langle \Ical u_{h,k} \, , P_{h}w_{u|u|^2} - w_{u|u|^2} \rangle \Big)  }\\
\nonumber  & \lesssim  & |\lambda_h -\lambda|  \, \|u_{h,k}-u\|_{\0} \|w_{u|u|^2} \|_{L^{2}(\D)} + h^{k} \|w_{u|u|^2}\|_{H^{k+1}(\D)} 
\,\,\, \overset{\eqref{suboptimal-estimate}}{\lesssim} \,\,\, h^{2k+1} \, .
\end{eqnarray} 
Combining everything yields $ |\lambda_{h,k} - \lambda | \lesssim h^{2k}$ and finishes the proof.
\end{proof}

\medskip
$\\$
{\bf Acknowledgements.}
The authors would like to thank the anonymous referees for their very insightful comments that greatly improved the contents of this paper.


\def\cprime{$'$}

\appendix \section{Appendix: inf-sup stability of $E''(u) -\lm \mathcal{I}$}
\label{appendix-inf-sup-stability}

In the following we prove the inf-sup stability of the bilinear form $E''(u) -\lm \mathcal{I}$ on $\5$ and $ V_h \cap \5 $. The proofs follow standard arguments in this context. We start with  Proposition \ref{prop-inf-sup-stability-Wperpu} which states that quasi-isolation in the sense of \ref{A5} implies inf-sup stability on $\5$.
\begin{proof}[Proof of Proposition \ref{prop-inf-sup-stability-Wperpu}]
Fix a ground state pair $(\lm , u) \in \R_{>0} \times H_{0}^{1}(\D)$ of \eqref{eigen_value_problem_1} and let $\gamma: (-\varepsilon,\varepsilon) \rightarrow \mathbb{S}$ be a smooth curve in the sense of Definition \ref{definition-quasi-isolated-ground-state} on the $L^{2}$-unit sphere $\mathbb{S}$ with $\gamma(0)=u \in \mathcal{U}$ and $\gamma^{\prime}(0)=w \in (\tangentspace{u} \cap \tangentspace{\ci u}) \setminus \{ 0 \} = 
\mathscr{H}_{u }^{\mathbb{C} , \perp} \setminus \{ 0 \} $. We define a Lagrange function  for the constrained minimization problem \eqref{definition-groundstate} as $L(v):= E(v) - \lm M(v)$ for $v \in \3$. A simple calculation gives $$ \frac{ \partial ^{2} }{\partial t^{2}} E(\gamma(t)) |_{t=0} = \langle L''(u) w , w \rangle = \langle (E''(u) - \lm \mathcal{I})w,w \rangle.$$ 
Now since $u$ is a quasi-isolated ground state of the energy functional $E$, we have
 \begin{align}
 \label{requirement1}
  \langle (E''(u) - \lm \mathcal{I})w,w \rangle = \frac{ \partial ^{2} }{\partial t^{2}} E(\gamma(t)) |_{t=0} > 0 \qquad \mbox{ for all } w \in \5 \setminus \{ 0\}. 
  \end{align}
Hence $0$ is not a eigenvalue of  $E''(u) - \lm \mathcal{I}$ on $\5$. Moreover, for any $w \in \5$, we have 
\begin{eqnarray}
\nonumber \lefteqn{\langle (E''(u) - \lm \mathcal{I})w,w \rangle \,\,\, = \,\,\,  \langle (\mathcal{A}_{|u|}-\lm \mathcal{I})w,w \rangle + 2\beta \int_{\D} (\re(u\overline{w}))^{2} \dx } \\
\label{requirement2} &\geq &  \alpha \|w\|^{2}_{\1} - \lm \|w\|^2_{\0} \, .\hspace{200pt} 
\end{eqnarray}
Thus the injectivity implied by \eqref{requirement1} together with the G{\aa}rding inequality \eqref{requirement2} yield the existence of a constant $\alpha_1 >0$ such that
\begin{align}
 \label{inf-sup-in-appendix}
\sup_{v \in \5} \frac{\langle (E''(u) - \lm \mathcal{I})v,w \rangle }{\|v\|_{\1}} \geq \alpha_1 \|w\|_{\1} \qquad \mbox{for all }  w \in \5 \, .
\end{align}
\end{proof}

%
With the inf-sup stability on $\5$, we can conclude the inf-sup stability of $E''(u)-\lambda \mathcal{I}$ on $ V_h \cap \5 $ for sufficiently small $h$.

\begin{proposition}[Discrete inf-sup stability]
\label{discrete-inf-sup-stability}
Assume \ref{A1}-\ref{A5}. Let $u \in \mathcal{U}$ denote a ground state in the sense of \eqref{definition-groundstate} and let $\lambda \in \R_{>0}$ denote the corresponding ground state eigenvalue given by \eqref{eigen_value_problem_2}. If $h$ is sufficiently small, then there is a constant $\alpha_2 >0$ independent of $h$, such that the discrete inf-sup condition is valid, namely
$$ \inf_{w \in  V_h \cap \5 } \sup_{v\in  V_h \cap \5  }\frac{|\langle(E''(u)-\lambda \mathcal{I})v,w\rangle|}{\|v\|_{\1}\hspace{2mm}\|w\|_{\1}} \hspace{2mm} \geqslant \hspace{2mm} \alpha_{2}.$$
\end{proposition}
\begin{proof}
The proof is standard and follows similarly as in \cite{Mel95} for the Helmholtz equation. 
Let $ w \in  V_h \cap \5  $ be arbitrary but fixed. For $z \in \5$ we consider the expression
\begin{eqnarray}
\label{proof-discrete-inf-sup-1}
\langle(E''(u)-\lambda \mathcal{I})(w+z),w\rangle &=& \langle E''(u)w,w\rangle +\langle (E''(u)-\lambda \mathcal{I})z,w\rangle - \lambda \langle \mathcal{I} w,w \rangle.
\end{eqnarray}
By the continuous inf-sup stability \eqref{inf-sup-in-appendix} (implied by \ref{A5})
there exists $z \in \5$ with
\begin{align}
 \label{proof-discrete-inf-sup-2}
 \langle (E''(u)-\lambda \mathcal{I})z,w\rangle = \lambda \langle \mathcal{I} w,w \rangle 
 \qquad \text{ and } \qquad \|z\|_{\1} \leq  C\|w\|_{\0}
\end{align}
and, as before, it holds the regularity estimate $\|z\|_{H^2(\D)} \lesssim \|w\|_{\0}$. 
Plugging \eqref{proof-discrete-inf-sup-2} in \eqref{proof-discrete-inf-sup-1} we get $\langle(E''(u)-\lambda \mathcal{I})(w+z),w\rangle = \langle E''(u)w,w\rangle $ and hence with the coercivity of $E''(u)$:
\begin{equation}
\label{proof-discrete-inf-sup-3}
\langle(E''(u)-\lambda \mathcal{I})(w+z),w\rangle = \langle E''(u)w,w\rangle \geq \alpha \|w\|_{\1}^2
\end{equation}
Now, we can choose an approximation  $z_h \in   V_h \cap \5 $ of $z$ such that
\begin{equation}
\label{proof-discrete-inf-sup-4}
\|z-z_h\|_{\1} \leq C_1\, h \, |z|_{H^2(\D)},
\end{equation}
where we used Lemmas \ref{projection-property-lemma} and \ref{H2-lemma} here. Therefore, we have
\begin{eqnarray}
\nonumber \lefteqn{ \langle(E''(u)-\lambda \mathcal{I})(w+z_h),w\rangle  \,\,\, = \,\,\,  \langle  ( E''(u)- \lambda \Ical )(w+z),w\rangle  -  \langle ( E''(u) - \lambda \Ical )(z-z_h),w\rangle  } \\
\nonumber  & \geq & \alpha \, \|w\|_{\1}^{2} - C_2 \|z-z_h\|_{\1} \|w\|_{\1} \\
 \nonumber & \geq & \alpha \, \|w\|_{\1}^{2} - C_3 h \, |z|_{H^2(\D)} \|w\|_{\1} \\
  \label{proof-discrete-inf-sup-5}  & \geq & \alpha \, \|w\|_{\1}^{2} - C_4 h \,  \|w\|_{\1}^2 
  \,\,\, \geq \,\,\, \tilde{\alpha} \|w\|_{\1}^2,\hspace{150pt}
\end{eqnarray}
 where  $\tilde{\alpha} >0$, if $h$ is sufficiently small. We conclude
\begin{eqnarray}
\nonumber \lefteqn{ \|w + z_h\|_{\1} \,\,\,\leq \,\,\, \|w\|_{\1} +\|z\|_{\1} +\|z-z_h\|_{\1} } \\
\label{proof-discrete-inf-sup-6} &\lesssim & \|w\|_{\1} + C\|w\|_{\1} +C_1 \, h \|z\|_{H^2(\D)}  \,\,\,\leq \,\,\, C_5 \|w\|_{\1} \, .
\end{eqnarray}
 Plugging \eqref{proof-discrete-inf-sup-6} in \eqref{proof-discrete-inf-sup-5}, we get
$$\langle(E''(u)-\lambda \mathcal{I})(w+z_h),w\rangle \geq \tilde{\alpha} \|w\|_{\1}^2 \geq \alpha_{2} \|w+z_h\|_{\1} \|w\|_{\1},$$
which implies
$$ \sup_{v \in  V_h \cap \5 } \frac{\langle(E''(u)-\lambda \mathcal{I})v,w\rangle }{\|v\|_{\1} \|w\|_{\1}} \geq   \frac{\langle(E''(u)-\lambda \mathcal{I}(w+z_h),w\rangle }{\|w+z_h\|_{\1} \|w\|_{\1}} \geq \alpha_{2}$$
Since $w\in \5 \cap V_{h}$ was arbitrary, we can take infimum to obtain the desired result.
\end{proof}

\end{document}